\RequirePackage{amsthm}

\documentclass[sn-mathphys,Numbered]{sn-jnl}
\usepackage{amsmath}
\usepackage{amssymb}
\usepackage{hyperref}
\usepackage{algorithm}
\usepackage{booktabs}
\usepackage{lineno} 
\usepackage{graphicx}
\usepackage{tabularx}
\usepackage{manyfoot}

\newcommand{\bA}{\mathbf A}
\newcommand{\bB}{\mathbf B}
\newcommand{\bD}{\mathbf D}
\newcommand{\bE}{\mathbf E}
\newcommand{\bF}{\mathbf F}
\newcommand{\bG}{\mathbf G}
\newcommand{\bH}{\mathbf H}
\newcommand{\bJ}{\mathbf J}
\newcommand{\bK}{\mathbf K}
\newcommand{\bL}{\mathbf L}
\newcommand{\bLambda}{\mathbf \Lambda}
\newcommand{\bP}{\mathbf P}
\newcommand{\bPi}{\mathbf \Pi}
\newcommand{\bQ}{\mathbf Q}
\newcommand{\bR}{\mathbf R}
\newcommand{\bS}{\mathbf S}
\newcommand{\bSigma}{\mathbf \Sigma}
\newcommand{\bT}{\mathbf T}
\newcommand{\bU}{\mathbf U}
\newcommand{\bV}{\mathbf V}
\newcommand{\bY}{\mathbf Y}
\newcommand{\ahess}{\bA_{+}}
\newcommand{\bxi}{\boldsymbol \xi}

\newcommand{\bb}{\mathbf b}
\newcommand{\bc}{\mathbf c}
\newcommand{\be}{\mathbf e}
\newcommand{\bg}{\mathbf g}
\newcommand{\bfr}{\mathbf r} 
\newcommand{\bs}{\mathbf s}

\newcommand{\bx}{\mathbf x}
\newcommand{\by}{\mathbf y}

\newcommand{\bz}{\mathbf z}

\newcommand{\calo}{\mathcal O}
\newcommand{\cals}{\mathcal S}
\newcommand{\caly}{\mathcal Y}

\newcommand{\calv}{\mathcal V}
\newcommand{\calw}{\mathcal W}

\newcommand{\tr}{\mathop{\mbox{tr}}}

\newcommand{\ds}{\displaystyle}
\newcommand{\wh}{\widehat}
\newcommand{\wt}{\widetilde}
\newcommand{\ph}{\phantom}
\DeclareMathOperator*{\argmin}{argmin}
\DeclareMathOperator*{\sym}{sym}

\newcommand{\step}[2]{\beta_{#2}^{\rm{#1}}}
\newcommand{\invstep}[2]{\alpha_{#2}^{\rm{#1}}} 

\newcommand{\extmat}[2]{[\,{#1} \ \ {#2}\,]}
\newcommand{\abbmin}{\text{ABB}_{\text{min}}}
\newcommand{\abbbon}{\text{ABB}_{\text{bon}}}

\usepackage{amsthm}
\theoremstyle{thmstyleone}%
\newtheorem{theorem}{Theorem}
\newtheorem{proposition}[theorem]{Proposition}%

\theoremstyle{thmstyletwo}%

\usepackage{xcolor}

\usepackage[normalem]{ulem}

\begin{document}

\title{Limited memory gradient methods for unconstrained optimization}

\author*[1]{\fnm{Giulia} \sur{Ferrandi}}\email{g.ferrandi@tue.nl}
\author[1]{\fnm{Michiel E.} \sur{Hochstenbach}}\email{m.e.hochstenbach@tue.nl}

\affil[1]{\orgdiv{Department of Mathematics and Computer Science}, \orgname{TU Eindhoven}, \orgaddress{\street{PO Box 513}, \city{Eindhoven}, \postcode{5600 MB}, \country{The Netherlands}}}

\abstract{
The limited memory steepest descent method (LMSD, Fletcher, 2012) for unconstrained optimization problems stores a few past gradients to compute multiple stepsizes at once.
We review this method and propose new variants. For strictly convex quadratic objective functions, we study the numerical behavior of different techniques to compute new stepsizes. In particular, we introduce a method to improve the use of harmonic Ritz values. We also show the existence of a secant condition associated with LMSD, where the approximating Hessian is projected onto a low-dimensional space.
In the general nonlinear case, we propose two new alternatives to Fletcher's method: first, the addition of symmetry constraints to the secant condition valid for the quadratic case; second, a perturbation of the last differences between consecutive gradients, to satisfy multiple secant equations simultaneously. We show that Fletcher's method can also be interpreted from this viewpoint.
}

\keywords{
limited memory steepest descent, 
unconstrained optimization, 
secant condition, 
low-dimensional Hessian approximation, 
Rayleigh--Ritz extraction,
Lyapunov equation.
}

\pacs[AMS Classification]{65K05, 90C20, 90C30, 65F15, 65F10}

\maketitle

\section{Introduction}
\label{sec:lmsd-intro}
We study the limited memory steepest descent method (LMSD), introduced by Fletcher \cite{fletcher2012limited}, in the context of unconstrained optimization problems for a continuously differentiable function $f$:
\[
\min_{\bx \in \mathbb{R}^n} f(\bx).
\]
The iteration for a steepest descent scheme reads
\begin{equation*} 
\bx_{k+1} = \bx_k - \beta_k \, \bg_k \ = \ \bx_k - \alpha_k^{-1} \, \bg_k,
\end{equation*}
where $\bg_k = \nabla f(\bx_k)$ is the gradient, $\beta_k>0$ is the steplength, and its inverse $\alpha_k = \beta_k^{-1}$ is usually chosen as an approximate eigenvalue of an (average) Hessian.
We refer to \cite{daniela2018steplength, zou2022delayed} for recent reviews on various steplength selection procedures.

The key idea of LMSD is to store the latest $m > 1$ gradients, and to compute (at most) $m$ new stepsizes for the following iterations of the gradient method. We first consider the strictly convex quadratic problem
\begin{equation} \label{quad}
\min_{\bx\in\mathbb{R}^n} \ \tfrac12 \, \bx^T\!\bA\bx -\bb^T\bx
\end{equation}
where $\bA$ is a symmetric positive definite (SPD) matrix with eigenvalues $0 < \lambda_1 \le \dots \le \lambda_n$, and $\bb \in \mathbb{R}^n$. Fletcher points out that the $m$ most recent gradients $\bG = [\, \bg_1 \ \dots \ \bg_m\,]$ form a basis for an $m$-dimensional Krylov subspace of $\bA$. (Although $\bG$ will change during the iterations, for convenience, and without loss of generality, we label the first column as $\bg_1$.)
Then, $m$ approximate eigenvalues of $\bA$ (Ritz values) are computed from the low-dimensional representation of $\bA$, a projected Hessian matrix, in the subspace spanned by the columns of $\bG$, and used as $m$ inverse stepsizes. For $m = 1$, the proposed method reduces to the steepest descent method with Barzilai--Borwein stepsizes \cite{bb1988}. 

LMSD shares the property with L-BFGS (see, e.g., \cite[Ch.~7]{nocedal2006numerical}),  the limited memory version of BFGS, that $2m$ past vectors are stored, of the form $\bs_{k-1} = \bx_k-\bx_{k-1}$ and $\by_{k-1} = \bg_k-\bg_{k-1}$. While LMSD is a first-order method which incorporates some second-order information in its simplest form (the stepsize), L-BFGS is a quasi-Newton method, which exploits the $\bs$-vectors and $\by$-vectors to provide an additive rank-$m$ update of a tentative approximate inverse Hessian (typically a multiple of the identity matrix). Compared to BFGS, at each iteration, the L-BFGS method computes the action of the approximate inverse Hessian, without storing the entire matrix and using $\calo(mn)$ operations (see, e.g., \cite[Ch.~7]{nocedal2006numerical}). As we will see in Section~\ref{sec:lmsd-exp}, the cost of $m$ LMSD iterations is approximately $\calo(m^2n)$, meaning that the costs of the two algorithms are comparable. 

There are several potential benefits of LMSD. First, as shown in \cite{fletcher2012limited}, there are some problems for which LMSD performs better than L-BFGS. Secondly, to the best of our knowledge and as stated in \cite[Sec.~6.4]{nocedal2006numerical}, there are no global convergence results for quasi-Newton methods applied to non-convex functions. Liu and Nocedal \cite{liu1989limited} have proved the global superlinear convergence of L-BFGS only for (twice continuously differentiable) uniformly convex functions. On the contrary, as a gradient method endowed with line search, LMSD converges globally for continuously differentiable functions (see \cite[Thm.~2.1]{raydan1997barzilai}, for the convergence of gradient methods combined with nonmonotone line search).
Finally, and quite importantly, we note that the idea of LMSD can be readily extended to other types of problems: to name a few, it has been used in the scaled spectral projected gradient method \cite{porta2015} for constrained optimization problems, in a stochastic gradient method \cite{franchini2020ritz}, and, more recently, in a projected gradient method for box-constrained optimization \cite{crisci2023hybrid}.




{\bf Summary of the state of the art and our contributions.} In the quadratic case \eqref{quad}, the projected Hessian matrix can be computed from the Cholesky decomposition of $\bG^T\bG$ (cf.~\cite[Eq.~(19)]{fletcher2012limited} and Section~\ref{sec:eig}) without involving any extra matrix-vector product with $\bA$. Although this procedure is memory and time efficient, it is also known to be potentially numerically unstable (cf., e.g., the discussion in \cite{FKN20}) because of the computation of the Gramian matrix $\bG^T\bG$, especially in our context of having an ill-conditioned $\bG$. 
Therefore, we consider alternative ways to obtain the projected Hessian in Section~\ref{sec:ritz}; in particular, we propose to use the pivoted QR decomposition of $\bG$ (see, e.g., \cite[Algorithm~1]{gu1996efficient}), or its SVD, and compare the three methods. 

In addition, we show that, in the quadratic case, there is a least squares secant condition associated with LMSD.
Indeed, in Section~\ref{sec:sec} we prove that the projected Hessian, obtained via one of these three decompositions, is similar to the solution to $\min_\bB \, \|\bY - \bS\bB\|$, where $\|\cdot\|$ denotes the Frobenius norm of a matrix, and $\bS = [\, \bs_1 \ \dots \ \bs_m\,]$ and $\bY = [\, \by_1 \ \dots \ \by_m\,]$ store the $m$ most recent $\bs$-vectors and $\by$-vectors, respectively.

Since $\bY = \bA\bS$ for quadratic functions, the obtained stepsizes are inverse eigenvalues of a projection of the Hessian matrix $\bA$.
In the general nonlinear case (i.e., for a non-quadratic function $f$), one can still reproduce the small matrix in \cite[Eq.~(19)]{fletcher2012limited}, since the Hessian is not needed explicitly in its computation.
However, there is generally not a clear interpretation of the stepsizes as approximate inverse eigenvalues of a certain Hessian matrix. Also, the obtained eigenvalues might even be complex.

To deal with this latter problem, Fletcher proposes a practical symmetrization of \cite[Eq.~(19)]{fletcher2012limited}, but, so far, a clear \emph{theoretical justification} for this approach seems to be lacking.
To address this issue, we rely on Schnabel's theorem \cite[Thm.~3.1]{schnabel1983quasi} to connect Fletcher's symmetrization to a perturbation of the $\bY$ matrix, of the form $\wt\bY = \bY + \Delta \bY$. This guarantees that the eigenvalues of the symmetrized matrix \cite[Eq.~(19)]{fletcher2012limited} correspond to a certain symmetric matrix $\bA_+$ that satisfies multiple secant equations $\wt\bY = \bA_+ \bS$ as in the quadratic case. The matrix $\bA_+$ can be interpreted as an approximate Hessian in the current iterate.

In the same line of thought, we also exploit one of the perturbations $\wt\bY$ proposed by Schnabel \cite{schnabel1983quasi} in the LMSD context. Although the idea of testing different perturbations of $\bY$ is appealing, a good perturbation may be expensive to compute, compared to the task of getting $m$ new stepsizes. Therefore, we explore a different approach based on the modification of the least squares secant condition of LMSD. The key idea is to add a {\em symmetry constraint} to the secant condition:
\[
\min_{\bB = \bB^T} \|\bY - \bS\bB\|. 
\]
Interestingly, the solution to this problem corresponds to the solution of a \emph{Lyapunov equation} (see, e.g., \cite{simoncini2016computational}). 
This secant condition provides a smooth transition from the strictly convex quadratic case to the general case, and its solution has real eigenvalues by construction. 

Along with discussing both the quadratic and the general case, we study the computation of \emph{harmonic Ritz values}, which are also considered by Fletcher \cite{fletcher2012limited} and Curtis and Guo \cite{curtis2016handling, curtis2018linear}.
For the quadratic case, in Section~\ref{sec:harm}, we show that there are some nice symmetries between the computation of the Ritz values of $\bA$ by exploiting a basis for the matrix of gradients $\bG$, and the computation of the {\em inverse} harmonic Ritz values of $\bA$ by means of $\bY$. Our implementation is different from Fletcher's, but the two approaches show similar performance in the quadratic experiments of Section~\ref{sec:lmsd-expquad}. In general, LMSD with harmonic Ritz values appears to show a less favorable behavior than LMSD with Ritz values. Therefore, in Section~\ref{sec:harm}, we present a way to improve the quality of the harmonic Ritz values, by taking an extra Rayleigh quotient of the harmonic Ritz vectors. This is based on the remarks in, e.g., \cite{morgan1991computing,sleijpen2003use}.

{\bf Outline.} The rest of the paper is organized as follows.
We first focus on the strictly convex quadratic problem \eqref{quad} in Section~\ref{sec:eig}. We review the LMSD method, as described by Fletcher \cite{fletcher2012limited}, and present new ways to compute the approximate eigenvalues of the Hessian. We also give a secant condition for the low-dimensional Hessian of which we compute the eigenvalues. We move to the general unconstrained optimization problems in Section~\ref{sec:genf}, where we give a theoretical foundation to Fletcher's symmetrized matrix \cite[Eq.~(19)]{fletcher2012limited}, and show how to compute new stepsizes from the secant equation for quadratics, by adding symmetry constraints. A third new approach based on \cite{schnabel1983quasi} is also proposed.
In both Sections~\ref{sec:eig} and \ref{sec:genf}, particular emphasis is put on the issue of (likely) numerical rank-deficiency of $\bG$ (or $\bY$, when computing the harmonic Ritz values). Sections~\ref{sec:algo-quad} and \ref{sec:algo-general} report the LMSD algorithms for strictly convex quadratic problems, as in \cite{fletcher2012limited}, and for general continuously differentiable functions, as in \cite{daniela2018steplength}.
Related convergence results are also recalled. Finally, numerical experiments on both strictly convex quadratics and general unconstrained problems are presented in Section~\ref{sec:lmsd-exp}; conclusions are drawn in Section~\ref{sec:con}.

Throughout the paper, the Frobenius norm of a matrix is denoted by $\|\cdot\|$. The eigenvalues of a symmetric matrix $\bA$ are ordered increasingly $\lambda_1\le\dots\le\lambda_n$, while its singular values are ordered decreasingly $\sigma_1\ge\dots\ge\sigma_n$.

\section{Limited memory BB1 and BB2 for quadratic problems} \label{sec:eig}
We review Fletcher's limited memory approach \cite{fletcher2012limited} for strictly convex quadratic functions \eqref{quad}, and study some new theoretical and computational aspects.
Common choices for the steplength in gradient methods for quadratic functions are the Barzilai--Borwein (BB) stepsizes \cite{bb1988}
\begin{equation} \label{eq:bbsteps}
\step{BB1}{k} = \frac{\bg_{k-1}^T \bg_{k-1}}{\bg_{k-1}^T \bA\, \bg_{k-1}}, \qquad
\step{BB2}{k} = \frac{\bg_{k-1}^T \bA\, \bg_{k-1}}{\bg_{k-1}^T \bA^2\, \bg_{k-1}}.
\end{equation}
The inverse stepsizes $\invstep{BB1}{k} = (\step{BB1}{k})^{-1}$ and $\invstep{BB2}{k} = (\step{BB2}{k})^{-1}$ are the standard and the harmonic Rayleigh quotients of $\bA$, evaluated at $\bg_{k-1}$, respectively. Therefore, they provide estimates of the eigenvalues of $\bA$.
The key idea of LMSD is to produce $m > 1$ approximate eigenvalues from an $m$-dimensional space simultaneously, hopefully capturing more information compared to that from a one-dimensional space. One hint about why considering $m > 1$ may be favorable is provided by the well-known Courant--Fischer Theorem and Cauchy's Interlace Theorem (see, e.g., \cite[Thms.~10.2.1 and 10.1.1]{Par98}). For two subspaces $\calv$, $\calw$ with $\calv \subseteq \calw$, we have
\[
\max_{\bz \in \calv, \, \|\bz\|=1} \bz^T\!\bA\bz \le \max_{\bz \in \calw, \, \|\bz\|=1} \bz^T\!\bA\bz \le
\max_{\|\bz\|=1} \bz^T\!\bA\bz = \lambda_n.
\]
Therefore, a larger search space may result in better approximations to the largest eigenvalue of $\bA$.
Similarly, a larger subspace may better approximate the smallest eigenvalue, as well as the next-largest and the next-smallest values. 

We now show why $m$ consecutive gradients form a basis of a Krylov subspace of $\bA$. It is easy to check that, given the stepsizes $\beta_1,\dots,\beta_{m}$ corresponding to the $m$ most recent gradients, each gradient can be expressed as follows:
\begin{equation} \label{eq:recursivegrad}
\bg_k = \prod_{i=1}^{k-1} \, (I - \beta_i\bA)\, \bg_1, \quad k = 1, \dots, m.
\end{equation}
Therefore all $m$ gradients belong to the Krylov subspace of degree $m$ (and of dimension at most $m$)
\[
\bg_k \in \mathcal{K}_m(\bA,\bg_1) = \text{span}\{\bg_1, \, \bA\, \bg_1, \, \dots, \, \bA^{m-1}\bg_1\}.
\]
Moreover, under mild assumptions, the columns of $\bG$ form a basis for $\mathcal{K}_m(\bA,\bg_1)$. This result is mentioned by Fletcher \cite{fletcher2012limited}; here we provide an explicit proof.

\begin{proposition} \label{prop:li}
Suppose the gradient $\bg_1$ does not lie in an $\ell$-dimensional invariant subspace, with $\ell < m$, of the SPD matrix $\bA$. 
If $\beta_k \ne 0$ for all $k = 1,\dots,m-1$, the vectors $\bg_1,\dots,\bg_m$ are linearly independent.
\end{proposition}
\begin{proof}
In view of the assumption, the set $\{\bg_1,A\, \bg_1\dots, A^{m-1}\bg_1\}$ is a basis for $\mathcal{K}_m(A,\bg_1)$. In fact, from \eqref{eq:recursivegrad},
\begin{equation}
\label{eq:li}
[\, \bg_1\ \ \bg_2\ \dots\ \ \bg_m\,] = [\, \bg_1\ \, A\bg_1\ \dots\ \, A^{m-1}\bg_1]
\begin{bmatrix}
1& \times & \times & \times & \times \\
 & -\beta_1& \times & \times & \times \\
 & &\beta_1\beta_2& \times & \times \\
 & & & \ddots & \times \\
 & & & &(-1)^{m}\prod_{i=1}^{m-1}\beta_i\\
\end{bmatrix}.
\end{equation}
Up to a sign, the determinant of the rightmost matrix in this equation is $\beta_1^{m-1}\beta_2^{m-2}\cdots\beta_{m-1}$, which is nonzero if and only if the stepsizes are nonzero. Therefore, $\bg_1, \dots, \bg_m$ are linearly independent.
\end{proof}

This result shows that $m$ consecutive gradients of a quadratic function are linearly independent in general; in practice, this formula suggests that small $\beta_i$ may quickly cause ill conditioning. Numerical rank-deficiency of $\bG$ is an important issue in the LMSD method and will be considered in the computation of a basis for $\text{span}(\bG)$ in Section~\ref{sec:ritz}.


For the following discussion, we also relate $\bS$ and $\bY$ to the Krylov subspace $\mathcal{K}_m(\bA,\bg_1)$.
\begin{proposition}
\label{prop:bases}
If $\bG$ is a basis for $\mathcal{K}_m(A,\bg_1)$, then
\begin{itemize}
\item[(i)] the columns of $\bS$ also form a basis for $\mathcal{K}_m(\bA,\bg_1)$;
\item[(ii)] the columns of $\bY$ form a basis for $\bA\, \mathcal{K}_m(\bA,\bg_1)$.
\end{itemize}
\end{proposition}
\begin{proof}
The thesis immediately follows from the relations
\begin{equation} \label{eq:sy-and-g}
\bS = -\bG\bD^{-1}, \quad \bY = -\bA\bG\bD^{-1}, \qquad \bD = \text{diag}(\alpha_1,\dots,\alpha_m),
\end{equation}
where the $\alpha_i = \beta_i^{-1}$ are the latest $m$ inverse stepsizes, ordered from the oldest to the most recent. Note that $\bD$ is nonsingular.
\end{proof}

Given a basis for $\mathcal{K}_m(\bA,\bg_1)$ (or $\bA\, \mathcal{K}_m(\bA,\bg_1)$), one can approximate some eigenpairs of $\bA$ from this subspace. The procedure is known as the Rayleigh--Ritz extraction method (see, e.g., \cite[Sec.~11.3]{Par98}) and is recalled in the next section.

\subsection{The Rayleigh--Ritz extraction}
\label{sec:ritz}
We formulate the standard and harmonic Rayleigh--Ritz extractions in the context of LMSD methods for strictly convex quadratic functions. 
Let $\cals$ be the subspace spanned by the columns of $\bS$, and $\caly$ be the subspace spanned by the columns of $\bY$. Fletcher's main idea \cite{fletcher2012limited} is to exploit the Rayleigh--Ritz method on the subspace $\cals$. We will now review and extend this approach.

We attempt to extract $m$ promising approximate eigenpairs from the subspace $\cals$. Therefore, such approximate eigenpairs can be represented as $(\theta_i, \bS \bc_i )$, with nonzero $\bc_i \in \mathbb{R}^m$, for $i = 1, \dots, m$. The (standard) Rayleigh--Ritz extraction imposes a Galerkin condition:
\begin{equation} \label{eq:ritz}
\bA\, \bS \, \bc - \theta \, \bS \, \bc \perp \cals.
\end{equation}
This means that the pairs $(\theta_i, \bc_i )$ are the eigenpairs of the $m \times m$ pencil $(\bS^T\bY, \, \bS^T\bS)$. The $\theta_i$ are called \emph{Ritz values}. In the LMSD method, we have $\cals = \mathcal{K}_m(\bA,\bg_1)$ (see Proposition~\ref{prop:bases}). Note that for $m = 1$, the only approximate eigenvalue reduces to the Rayleigh quotient $\alpha^{\rm BB1}$ \eqref{eq:bbsteps}. Ritz values are bounded by the extreme eigenvalues of $\bA$, i.e., $\theta_i \in [\lambda_1, \lambda_n]$. This follows from Cauchy's Interlace Theorem \cite[Thm.~10.1.1]{Par98}, by choosing an orthogonal basis for $\cals$. This inclusion is crucial to prove the global convergence of LMSD for quadratic functions \cite{fletcher2012limited}.

Although the matrix of gradients $\bG$ (or $\bS$) already provides a basis for $\cals$, from a numerical point of view it may not be ideal to exploit it to compute the Ritz values, since $\bG$ is usually numerically ill conditioned. 
Therefore, we recall Fletcher's approach \cite{fletcher2012limited} to compute a basis for $\cals$, and then propose two new variants: via a pivoted QR and via an SVD. Fletcher starts by a QR decomposition $\bG = \bQ\bR$, discarding the oldest gradients whenever $\bR$ is numerically singular.
Then $\bQ$ is an orthogonal basis for a possibly smaller space $\cals = \text{span}([\bg_{m-s+1},\dots,\bg_{m}])$, with $s \le m$.
The product $\bA\bG$ can be computed from the gradients without additional multiplications by $\bA$, in view of
\begin{equation} \label{eq:yg}
\bA\bG = -\bY\bD = \extmat{\bG}{\bg_{m+1}}\ \bJ, \qquad \text{where} \quad
\bJ = 
\begin{bmatrix}
\ph{-}\alpha_1 \\[-2mm]
-\alpha_1& \ddots \\[-1.5mm]
& \ddots & \ph{-}\alpha_m \\
&& -\alpha_m\\
\end{bmatrix}.
\end{equation}
Here, the relation $\by_{k-1} = \bg_k-\bg_{k-1}$ is used.
Then the $s\times s$ low-dimensional representation of $\bA$ can be written in terms of $\bR$:
\begin{equation} \label{eq:fletcherT}
\bT := \bQ^T\!\bA\bQ = \extmat{\bR}{\bfr} \ \bJ \, \bR^{-1},
\end{equation}
where $\bfr = \bQ^T\bg_{m+1}$.
It is clear that $\bT$ is symmetric; it is also tridiagonal in view of the fact that it is associated with a Krylov relation for a symmetric matrix (see also Fletcher \cite{fletcher2012limited}). 
Since $\bfr$ is also the solution to $\bR^T\bfr = \bG^T\bg_{m+1}$, the matrix $\bQ$ is in fact not needed to compute $\bfr$.
For this reason, Fletcher concludes that the Cholesky decomposition $\bG^T\bG = \bR^T\bR$  is sufficient to determine $\bT$ and its eigenvalues. Standard routines raise an error when $\bG^T\bG$ is numerically not SPD (numerically having a zero or tiny negative eigenvalue). If this happens, the oldest gradients are discarded (if necessary one by one in several steps), and the Cholesky decomposition is repeated.

Instead of discarding the oldest gradients $\bg_1, \dots, \bg_{m-s}$, we will now consider a new variant by selecting the gradients in the following way. We carry out a pivoted QR decomposition of $\bG$, i.e., $\bG\wh\bPi = \wh\bQ\wh\bR$, where $\wh\bPi$ is a permutation matrix that iteratively picks the column with the maximal norm after each Gram--Schmidt step \cite{gu1996efficient}. 
As a consequence, the diagonal entries of $\wh\bR$ are ordered nonincreasingly in magnitude.
(In fact, we can always ensure that these entries are positive, but since standard routines may output negative values, we consider the magnitudes.)

The pivoted QR approach is also a rank-revealing factorization, although generally less accurate than the SVD (see, e.g., \cite{gu1996efficient}). Let $\wh\bR_{G}$ be the first $s\times s$ block of $\wh\bR$ for which $\vert \wh r_{i}\vert > {\sf thresh}\cdot\vert \wh r_{1}\vert$, where $\wh r_i$ is the $i$th diagonal element of $\wh R$ and ${\sf thresh} > 0$. A crude approximation to its condition number is $\kappa(\wh\bR_{G}) \approx \vert \wh r_{1}\vert\,/\, \vert \wh r_{s}\vert$.
Although this approximation may be quite imprecise, the alternative to repeatedly compute $\kappa(\wh\bR_{G})$ by removing the last column and row of the matrix at each iteration might take up to $\calo(m^4)$ work, which, even for modest values of $m$, may be unwanted.

The approximation subspace for the eigenvectors of $\bA$ is now $\cals = \text{span}(\wh\bQ_G)$, with $\bG\wh\bPi_G = \wh\bQ_G\wh\bR_{G}$, where $\wh\bPi_G$ and $\wh\bQ_G$ are the first $s$ columns of $\wh\bPi$ and $\wh\bQ$, respectively.
The upper triangular $\wh\bR$ can be partitioned as follows: 
\begin{equation} \label{eq:Rpivo}
\wh\bR = 
\begin{bmatrix}
\wh\bR_{G} & \wh\bR_{12}\\ {\bf 0} & \wh\bR_{22} 
\end{bmatrix}.
\end{equation}
As in \eqref{eq:fletcherT}, we exploit \eqref{eq:yg} to compute the projected Hessian
\begin{align}
 \bB^{\rm QR} &:= \wh\bQ_G^T\, \bA\wh\bQ_G = \wh\bQ_G^T\, \bA\, \bG\, \wh\bPi_G \wh\bR_{G}^{-1} = \wh\bQ_G^T\, \extmat{\wh\bQ\wh\bR\wh\bPi^{-1}}{\bg_{m+1}}\, \bJ \wh\bPi_G \wh\bR_{G}^{-1}\nonumber\\
 &= \extmat{[\wh\bR_{G}\; \wh\bR_{12}]\, \wh\bPi^{-1}}{\wh\bQ_G^T\, \bg_{m+1}}\, \bJ \, \wh\bPi_G \, \wh\bR_{G}^{-1}.\label{eq:qrbb1}
\end{align}
Note that, compared to Fletcher's approach, this decomposition removes the unwanted gradients all at once, while in \cite{fletcher2012limited} 
the Cholesky decomposition is repeated every time the $\bR$ matrix is numerically singular. 
Fletcher's $\bT$ \eqref{eq:fletcherT} is a specific case of \eqref{eq:qrbb1}, where $\wh\bPi = \wh\bPi_G$ is the identity matrix, and $\wh\bR_{G}$ is the whole $\wh\bR$, but where $\bG$ only contains $[\bg_{m-s+1},\dots,\bg_{m}]$.

As the second new variant, we exploit an SVD decomposition $\bG = \bU \, \bSigma \, \bV^T$, where $\bSigma$ is $m\times m$, to get a basis for $\cals$.
An advantage of an SVD is that this provides a natural way to reduce the space by removing the singular vectors corresponding to singular values below a certain tolerance.
We decide to retain the $s\le m$ singular values for which $\sigma_i \ge {\sf thresh}\cdot \sigma_1$, where $\sigma_1$ is the largest singular value of $\bG$. Therefore we consider the truncated SVD $\bG \approx \bG_1 = \bU_G \, \bSigma_G \, \bV_G^T$, where the matrices on the right-hand side are $n \times s$, $s \times s$, and $s \times m$, respectively. Then the approximation subspace becomes $\cals = \text{span}(\bU_G)$, and we compute the corresponding $s\times s$ representation of $\bA$. Since $\bG_1\bV_G = \bG\bV_G$, and $\bU_G = \bG_1\bV_G\bSigma_G^{-1}$, we have, using the expression for $\bA\bG$ \eqref{eq:yg},
\begin{align}
\bB^{\rm SVD} &= \bU_G^T\, \bA\bU_G = \bU_G^T\, \bA \bG\, \bV_G\, \bSigma_G^{-1} = \bU_G^T\, \extmat{\bU \, \bSigma \, \bV^T}{\bg_{m+1}}\ \bJ \, \bV_G\, \bSigma_G^{-1}.\nonumber\\
& = \extmat{\bSigma_G \bV_G^T}{\ \bU_G^T\,\bg_{m+1}}\ \bJ\, \bV_G\, \bSigma_G^{-1}.\label{eq:svdbb1-g}
\end{align}
We remark that, by construction, both $\bB^{\rm SVD}$ and $\bB^{\rm QR}$ are SPD.
Due to the truncation of the decompositions of $\bG$ in both the pivoted QR and SVD techniques, the subspace $\cals$ will generally not be a Krylov subspace, in contrast to Fletcher's method. An immediate consequence of this is that, in contrast with $\bT$ in \eqref{eq:fletcherT}, the matrices $\bB^{\rm SVD}$ and $\bB^{\rm QR}$ are not tridiagonal. Still, of course, one can also expect to extract useful information from a non-Krylov subspace.

Since LMSD with Ritz values can be seen as an extension of a gradient method with BB1 stepsizes, it is reasonable to look for a limited memory extension of the gradient method with BB2 stepsizes. The harmonic Rayleigh--Ritz extraction is a suitable tool to achieve this goal.

\subsection{The harmonic Rayleigh--Ritz extraction}
\label{sec:harm}
The use of harmonic Ritz values in the context of LMSD has been mentioned by Fletcher \cite[Sec.~7]{fletcher2012limited}, and further studied by Curtis and Guo \cite{curtis2016handling}. While the Rayleigh--Ritz extraction usually finds good approximations for exterior eigenvalues, the harmonic Rayleigh--Ritz extraction has originally been introduced to approximate eigenvalues close to a target value in the interior of the spectrum. A natural way to achieve this is to consider a Galerkin condition for $\bA^{-1}$:
\begin{equation} \label{eq:harmritz}
\bA^{-1}\bY\wt \bc - \wt\theta^{-1} \bY\wt \bc \perp \caly,
\end{equation}
which leads to the eigenpairs $(\wt\theta_i^{-1}, \wt\bc_i )$ of the pair $(\bY^T\bS, \, \bY^T\bY)$. 
However, since $\bA^{-1}$ is usually not explicitly available or too expensive to compute, one may choose a subspace of the form $\caly = \bA\, \cals$ (see, e.g., \cite{morgan1991computing}). This simplifies the Galerkin condition:
\[
\bA\bS \, \wt\bc - \wt\theta \, \bS \, \wt\bc \perp \bA\, \cals.
\]
The eigenvalues $\wt \theta_i$ from this condition are called \emph{harmonic Ritz values}. 
In the limited memory extension of BB2 we set $\caly = \bA\, \mathcal{K}_m(\bA,\bg_1)$, and we know that $\bY$ is a basis for $\caly$ from Proposition~\ref{prop:bases}. Harmonic Ritz values are also bounded by the extreme eigenvalues of $\bA$: $\wt\theta_i \in [\lambda_1, \lambda_n]$; see, e.g., \cite[Thm.~2.1]{beattie1998harmonic}.
It is easy to check that the (memory-less) case $m = 1$ corresponds to the computation of the harmonic Rayleigh quotient $\alpha^{\rm BB2}$.

We have just observed that the Galerkin condition for the harmonic Ritz values can be formulated either in terms of $\bY$ or $\bS$. The latter way is presented in the references \cite{fletcher2012limited, curtis2016handling}, which again look for a basis of $\cals$ by means of a QR decomposition of $\bG$.
Following the line of \cite{fletcher2012limited}, the aim is to find the eigenvalues of 
\begin{equation}
\label{eq:harmfletcher}
 (\bQ^T\!\bA\bQ)^{-1}\, \bQ^T\!\bA^2\bQ =: \bT^{-1}\, \bP,
\end{equation}
where $\bG = \bQ\bR$. Since $\bQ^T\!\bA^2\bQ$ involves the product $\extmat{\bG}{\bg_{m+1}}^T \extmat{\bG}{\bg_{m+1}}$, we determine the Cholesky decomposition of this matrix, to write \cite[Eq.~(30)]{fletcher2012limited}
\begin{equation}
 \label{eq:fletcherP}
 \bP = \bR^{-T}\bJ^T\begin{bmatrix}
 \bR & \bfr \\ {\bf 0} & \rho
 \end{bmatrix}^T 
 \begin{bmatrix}
 \bR & \bfr \\ {\bf 0} & \rho 
 \end{bmatrix}\bJ\bR^{-1},
\end{equation}
where $\big[\begin{smallmatrix}
  \bR & \bfr\\
  {\bf 0} & \rho
\end{smallmatrix}\big]$ is the Cholesky factor of $\extmat{\bG}{\bg_{m+1}}^T \extmat{\bG}{\bg_{m+1}}$, such that $\bR$ is the Cholesky factor of $\bG^T\bG$, $\bfr = \bQ^T\bg_{m+1}$ as in \eqref{eq:fletcherT}, while $\rho$ is a scalar.
Both $\bT$ and $\bP$ are symmetric; moreover, while $\bT$ is tridiagonal, $\bP$ is pentadiagonal. If $\bG$ is rank deficient, the oldest gradients are discarded.

Given the similar roles of $\bS$ for $\bA$ in \eqref{eq:ritz} and of $\bY$ for $\bA^{-1}$ in \eqref{eq:harmritz}, we now consider new alternative ways to find the harmonic Ritz values of $\bA$, based on the decomposition of either $\bY$ or $\bY^T\bY$.
The aim is to get an $s\times s$ representation of $\bA^{-1}$, as we did for $\bA$ in Section~\ref{sec:ritz}. In this context, we need the following (new) relation:
\begin{equation}
\label{eq:gy}
\bA^{-1}\bY = -\bG\bD^{-1} = \extmat{\bY}{-\bg_{m+1}}\, \wt \bJ,\quad\text{where}\quad
\wt \bJ = 
\begin{bmatrix}
1& & & \\[-1.5mm]
\vdots & \ddots & \\
1 &\cdots& 1\\[-0.5mm]
1 &\cdots& 1\\
\end{bmatrix} \!\bD^{-1}.
\end{equation}
As for \eqref{eq:yg}, this follows from the definition $\by_{k-1} = \bg_k-\bg_{k-1}$.

We start with the pivoted QR of $\bY$, i.e., $\bY\check\bPi = \check\bQ\check\bR$. As in Section~\ref{sec:ritz}, we truncate the decomposition based on the diagonal values of $\check\bR$, and obtain  $\bY\check\bPi_Y = \check\bQ_Y\check\bR_Y$, with 
\begin{equation*} 
\check\bR = 
\begin{bmatrix}
\check\bR_{Y} & \check\bR_{12}\\ {\bf 0} & \check\bR_{22} 
\end{bmatrix}.
\end{equation*}

Then we project $\bA^{-1}$ onto $\caly = {\rm span}(\bQ_Y)$ to obtain
\begin{align}
 \bH^{\rm QR} &= \check\bQ_Y^T\bA^{-1}\check\bQ_Y = \check\bQ_Y^T\bA^{-1}\, \bY\, \check\bPi_Y \check\bR_Y^{-1} = \check\bQ_Y^T\,\extmat{\bY}{-\bg_{m+1}}\, \wt \bJ \, \check\bPi_Y \check\bR_Y^{-1}\nonumber\\
 &= \extmat{[\check\bR_Y\; \check\bR_{12}]\, \check\bPi^{-1}}{-\check\bQ_Y^T\bg_{m+1}}\, \wt \bJ \,\check\bPi_Y \check\bR_Y^{-1}.\label{eq:qrbb2}
\end{align}
The matrix $\bH^{\rm QR}$ is also symmetric and delivers the reciprocals of harmonic Ritz values; its expression is similar to \eqref{eq:qrbb1}.
An approach based on the Cholesky decomposition of $\bY^T\bY = \wt\bR^T\wt\bR$ may also be derived: 
\begin{equation}
 \label{eq:HYR}
 \bH^{\rm CH} = \extmat{\wt\bR}{\wt\bfr}\, \wt \bJ\, \wt\bR^{-1},
\end{equation}
with $\wt\bfr$ solution to $\wt\bR^T\wt\bfr = -\bY^T\bg_{m+1}$.

As for the Ritz values, SVD is another viable option. Consider the truncated SVD of $\bY$: 
$\bY_1 = \bU_Y \, \bSigma_Y \, \bV_Y^T$, where $\bSigma_Y$ is $s\times s$. Since $\bY_1\bV_Y = \bY\bV_Y$, by using similar arguments as in the derivation of \eqref{eq:svdbb1-g}, we get the following low-dimensional representation of $\bA^{-1}$:
\begin{align} 
\bH^{\rm SVD} &= \bU_Y^T\bA^{-1}\bU_Y = \bU_Y^T\bA^{-1} \bY\, \bV_Y\, \bSigma_Y^{-1} = \bU_Y^T\,\extmat{\bY}{-\bg_{m+1}}\, \wt \bJ\, \bV_Y\, \bSigma_Y^{-1} \nonumber\\
&= \extmat{\bSigma_Y \bV_Y^T}{-\bU_Y^T\bg_{m+1}}\, \wt \bJ\, \bV_Y\, \bSigma_Y^{-1}.\label{eq:svdbb2}
\end{align}
Note that, in contrast to $\bT^{-1}\bP$, the matrix $\bH^{\rm SVD}$ is symmetric and gives the reciprocals of harmonic Ritz values. In addition, the expression for $\bH^{\rm SVD}$ is similar to the one for $\bB^{\rm SVD}$ in \eqref{eq:svdbb1-g}.

To conclude the section, we mention the following technique, which is new in the context of LMSD.
For the solution of eigenvalue problems, it has been observed (e.g., by Morgan \cite{morgan1991computing}) that harmonic Ritz values sometimes do not approximate eigenvalues well, and it is recommended to use the Rayleigh quotients of harmonic Ritz vectors instead.
This means that we use $\bS\wt \bc_i$ as approximate eigenvectors, and their Rayleigh quotients $\wt \bc_i^T\bS^T\!\bA\bS\wt \bc_i$ as approximate eigenvalues.
This fits nicely with Fletcher's approach: in fact, once we have the eigenvectors $\wt \bc_i$ of $\bT^{-1}\bP$ \eqref{eq:harmfletcher}, we compute their corresponding Rayleigh quotients as $\wt \bc_i^T\bT \wt \bc_i$.
We remark that, in the one-dimensional case, this procedure reduces to the gradient method with BB1 stepsizes, instead of the BB2 ones.

In Section~\ref{sec:lmsd-expquad} we compare and comment on the different strategies to get both the standard and the harmonic Ritz values. We will see how the computation of the harmonic Rayleigh quotients can result in a lower number of iterations of LMSD, although computing extra Rayleigh quotients involves some additional work in the $m$-dimensional space.

\subsection{An algorithm for strictly convex quadratic functions}
\label{sec:algo-quad}
In this section we present the LMSD method for strictly convex quadratic functions. As already mentioned, the key idea of the algorithm is to store the $m$ most recent gradients or $\by$-vectors, to compute up to $s\le m$ new stepsizes, according to one of the procedures described in Sections~\ref{sec:ritz} and \ref{sec:harm}. These stepsizes are then used in (up to) $s$ consecutive iterations of a gradient method; this group of iterations is referred to as a \emph{sweep} \cite{fletcher2012limited}. 

In Algorithm~\ref{algo:lmsd}, we report the LMSD method for strictly convex quadratic functions as proposed in \cite[``A Ritz sweep algorithm"]{fletcher2012limited}. This routine is a gradient method without line search. Particular attention is put into the choice of the stepsize: whenever the function value increases compared to the initial function value of the sweep $f_{\rm ref}$, Fletcher resets the iterate and computes a new point by taking a Cauchy step (cf.~Line~9). This ensures that the next function value will not be higher than the current $f_{\rm ref}$, since the Cauchy step is the solution to the exact line search $\min_\beta f(\bx_k-\beta\,\bg_k)$. Additionally, every time we take a Cauchy step, or the norm of the current gradient has increased compared to the previous iteration, we clear the stack of stepsizes and compute new (harmonic) Ritz values. At each iteration, a new gradient or $\by$-vector is stored, depending on the method chosen to approximate the eigenvalues of $\bA$ (cf.~Sections~\ref{sec:ritz} and \ref{sec:harm}).

\begin{algorithm}[htb!] 
\caption{LMSD for strictly convex quadratic functions \cite{fletcher2012limited}}
{\bf Input}: Function $f(\bx) = \frac12 \, \bx^T\!\bA\bx-\bb^T\bx$ with $\bA$ SPD, initial guess $\bx_0$, initial stepsize $\beta_0 > 0$, tolerance {\sf tol} \\
{\bf Output}: Approximation to minimizer $\argmin_{\bx} f(\bx)$ \\
\begin{tabular}{ll}
{\footnotesize 1}: & $\bg_0 = \nabla f(\bx_0)$, \quad$f_{\rm ref} = f(\bx_0)$ \\
{\footnotesize 2}: & $j = 0$,\quad $s = 1$ \hfill {\# $s$ is the stack size}\\
{\footnotesize 3}: & {\bf for} $k = 0, 1, \dots$ \\
{\footnotesize 4}: & \ph{M} $\nu_k = \beta_j$, \quad $j = j + 1$\\
{\footnotesize 5}: & \ph{M} $\bx_{k+1} = \bx_k -\nu_k \, \bg_k$ \\
{\footnotesize 6}: & \ph{M} {\bf if} \ $\|\bg_{k+1}\| \le {\sf tol} \cdot \|\bg_0\|$, \ {\bf return}, \ {\bf end} \\
{\footnotesize 7}: & \ph{M} {\bf if} \ $f(\bx_{k+1}) \ge f_{\rm ref}$ \\
{\footnotesize 8}: & \ph{MM} Reset $\bx_{k+1} = \bx_k$, \ clear the stack\\
{\footnotesize 9}: & \ph{MM} Reset $\beta_0 = \bg_k^T\bg_k\,/\,\bg_k^T\bA\bg_k$, $j = 0$, $s=1$ \hfill {\# Cauchy stepsize} \\
{\footnotesize 10}: & \ph{MM} {\bf continue} \\
{\footnotesize 11}: & \ph{M} {\bf else} \\
{\footnotesize 12}: & \ph{MM} {\bf if} \ $\|\bg_{k+1}\| \ge \|\bg_k\|$,\ clear the stack, \ {\bf end}\\
{\footnotesize 13}: & \ph{M} {\bf end}\\
{\footnotesize 14}: & \ph{M} {\bf if} \ empty stack \ {\bf or} \ $j\ge s$\\
{\footnotesize 15}: & \ph{MM} Compute stack of $s\le m$ new stepsizes $\beta_j$, ordered increasingly \\
{\footnotesize 16}: & \ph{MM} $j=0$,\ $f_{\rm ref} = f(\bx_{k+1})$, \ {\bf end} \\
{\footnotesize 17}: & {\bf end} \\
\end{tabular}
\label{algo:lmsd}
\end{algorithm}

It is possible to implement LMSD without controlling the function value of the iterates or the gradient norm, as in \cite{curtis2018linear}. Here Curtis and Guo also show the R-linear convergence of the method. However, in our experiments, we have noticed that this latter implementation converges slower than Fletcher's (for quadratic problems).

The stepsizes are plugged in the gradient method in increasing order, but there is no theoretical guarantee that this choice is optimal in some sense. From a theoretical viewpoint, the ordering of the stepsizes is irrelevant in a gradient method for strictly convex quadratic functions, as is apparent from \eqref{eq:recursivegrad}.
In practice, due to rounding errors and other additions to the implementation (such as, e.g., Lines~7--13 of Algorithm~\ref{algo:lmsd}), the stepsize ordering is relevant for both the quadratic and the general nonlinear case, which will be discussed in Section~\ref{sec:algo-general}.
For the quadratic case, Fletcher \cite{fletcher2012limited} suggests that choosing the increasing order improves the chances of a monotone decrease in both the function value and the gradient norm. Nevertheless, his argument is based on the knowledge of $s$ exact eigenvalues of $\bA$ \cite{fletcher1990low}. 

To the best of our knowledge, an aspect that has not been discussed yet is the presence of rounding errors in the low-dimensional representation of the Hessian. Except for \eqref{eq:harmfletcher}, all the obtained matrices are symmetric, but their expressions are not.
Therefore, in a numerical setting, it might happen that a representation of the Hessian is not symmetric. This may result in negative or complex eigenvalues; for this reason, we enforce symmetry by taking the symmetric part of the projected Hessian, i.e., $\bB \leftarrow \frac12(\bB + \bB^T)$, which is the symmetric matrix nearest to $\bB$.
In the Cholesky decomposition, we replace the upper triangle of $\bT$ with the transpose of its lower triangle, in agreement with Fletcher's choice for the unconstrained case (cf.~\cite{fletcher2012limited} and Section~\ref{sec:schnabel}). In both situations, we discard negative eigenvalues, which may still arise.

In practice, we observe that the non-symmetry of a projected Hessian appears especially in problems with large $\kappa(\bA)$, for a relatively large choice of $m$ (e.g., $m = 10$) and a small value of ${\sf thresh}$ (e.g., ${\sf thresh} = 10^{-10}$). 
In this situation, the Cholesky decomposition seems to produce a non-symmetric projected Hessian more often than pivoted QR or SVD. This is likely related to the fact that the Cholesky decomposition of an ill-conditioned Gramian matrix leads to a more inaccurate $\bR$ factor (cf.~Section~\ref{sec:lmsd-intro}). In addition, the symmetrized $\bT$ seems to generate negative eigenvalues more often than the Hessian representations obtained via pivoted QR and SVD. However, these aspects may not directly affect the performance of LMSD. As we will see in Section~\ref{sec:lmsd-expquad}, the adoption of different decompositions does not seem to influence the speed of LMSD. 

We finally note that for smaller values of $m$, such as $m = 5$, the projected Hessian tends to be numerically symmetric even for a small ${\sf thresh}$. In fact, fewer gradients form a better condition matrix, because of the following argument. First, we have $\sigma_i^2(\bG) = \lambda_{m-i+1}(\bG^T\bG)$, for $i = 1,\dots,m$.
Since the Gramian matrix of the $s \le m$ most recent gradients $[\bg_{m-s+1},\dots,\bg_m]$ is a submatrix of $\bG^T\bG$, from Cauchy's Interlace Theorem (see, e.g., \cite[Thms.~10.2.1 and 10.1.1]{Par98}), we get that $\sigma_{\min}(\bG) \le \sigma_{\min}([\bg_{m-s+1},\dots,\bg_m])$ and $\sigma_{\max}(\bG) \ge \sigma_{\max}([\bg_{m-s+1},\dots,\bg_m])$. This proves that $\kappa([\bg_{m-s+1},\dots,\bg_m]) \le \kappa(\bG)$.

\subsection{Secant conditions for LMSD}
\label{sec:sec}
We finally show that the low-dimensional representations of the Hessian matrix $\bA$ (or its inverse) satisfy a certain secant condition. This result is new in the context of LMSD, and will be the starting point of one of our extensions of LMSD to general unconstrained optimization problems in Section~\ref{sec:genf}.
Recall from \cite{bb1988} that the BB stepsizes \eqref{eq:bbsteps} satisfy a secant condition each, in the least squares sense:
\[
\invstep{BB1}{} = \argmin_{\alpha} \, \|\by-\alpha\, \bs\|, \qquad \invstep{BB2}{} = \argmin_{\alpha} \, \|\alpha^{-1} \, \by-\bs\|.
\]
We now give a straightforward extension of these conditions to the limited memory variant of the steepest descent. We show that there exist $m\times m$ matrices that satisfy a secant condition and share the same eigenvalues as the two pencils $(\bS^T\bY, \, \bS^T\bS)$, $(\bY^T\bS, \, \bY^T\bY)$. In the quadratic case, when $\bY = \bA\bS$, the following results correspond to \cite[Thm.~11.4.2]{Par98} and \cite[Thm.~4.2]{vomel2010note}, respectively.
\begin{proposition}
\label{prop:ls}
Let $\bS$, $\bY\in \mathbb{R}^{n\times m}$ be full rank, with $n \ge m$, and let $\bB$, $\bH\in \mathbb{R}^{m\times m}$.
\begin{enumerate}
\item[(i)] The unique solution to $\ds \min_\bB \|\bY-\bS\bB\|$ is $\bB = (\bS^T\bS)^{-1} \bS^T\bY$.
\item[(ii)] The unique solution to $\ds \min_\bH \|\bY\bH-\bS\|$ is $\bH = (\bY^T\bY)^{-1} \bY^T\bS$.
\item[(iii)] In the quadratic case \eqref{quad}, the eigenvalues of $\bB$ are the Ritz values of $\bA$ and the eigenvalues of $\bH$ are the \emph{inverse} harmonic Ritz values of $\bA$.
\end{enumerate}
\end{proposition}
\begin{proof}
The stationarity conditions for the overdetermined least squares problem $\min_\bB \ \|\bY-\bS\bB\|$ are the normal equations $\bS^T(\bY - \bS\bB) = {\bf 0}$. Since $\bS$ is of full rank, $\bS^T\bS$ is nonsingular, and thus $\bB = (\bS^T\bS)^{-1} \, \bS^T\bY$. Part (ii) follows similarly, by exchanging the role of $\bS$ and $\bY$. Since $\bB$ and $(\bS^T\bY, \, \bS^T\bS)$ have the same eigenvalues, part (iii) easily follows. The same relation holds for the eigenvalues of $\bH$ and the eigenvalues of the pencil $(\bY^T\bS, \, \bY^T\bY)$.
\end{proof}
Proposition~\ref{prop:ls} is a good starting point to extend  LMSD for solving general unconstrained optimization problems. 

\section{General nonlinear functions} \label{sec:genf}
When the objective function $f$ is a general continuously differentiable function, the Hessian is no longer constant through the iterations, and not necessarily positive definite. In general, there is no SPD approximate Hessian such that multiple secant equations hold (that is, an expression analogous to $\bY = \bA\bS$ in the quadratic case).
This is clearly stated by Schnabel \cite[Thm.~3.1]{schnabel1983quasi}.
\begin{theorem}
\label{thm:sch}
Let $\bS,\, \bY$ be full rank. Then there exists a symmetric (positive definite) matrix $\ahess$ such that $\bY = \ahess\, \bS$ if and only if $\bY^T\bS$ is symmetric (positive definite).
\end{theorem}
By inspecting all the expressions derived in Sections~\ref{sec:ritz} and \ref{sec:harm}, we observe that only $\bG$ and $\bY$ are needed to compute the $m\times m$ matrices of interest for LMSD. However, given that $\bY^T\bS$ is in general not symmetric, Theorem~\ref{thm:sch} suggests that we cannot interpret these matrices as low-dimensional representations of some Hessian matrices.

We propose two ways to restore the connection with Hessian matrices.
In Section~\ref{sec:schnabel}, we exploit a technique proposed by Schnabel \cite{schnabel1983quasi} for quasi-Newton methods. It consists of perturbing $\bY$ to make $\bY^T\bS$ symmetric. We show that Fletcher's method can also be interpreted in this way.
In Section~\ref{sec:symsec}, we introduce a second method which does not aim at satisfying multiple secant equations at the same time, but finds the solution to the least squares secant conditions of Proposition~\ref{prop:ls} by imposing symmetry constraints. 

\subsection{Perturbation of $\bY$ to solve multiple secant equations} \label{sec:schnabel}
In the context of quasi-Newton methods, Schnabel \cite{schnabel1983quasi} proposes to perturb the matrix $\bY$ of a quantity $\Delta \bY = \wt \bY - \bY$ to obtain an SPD $\wt \bY^T\bS$.
With this strategy, we implicitly obtain a certain SPD approximate Hessian $\ahess$ such that $\wt \bY = \ahess\, \bS$. We then refer to Sections~\ref{sec:ritz} and \ref{sec:harm} to compute either the Ritz values or the harmonic Ritz values of the approximate Hessian $\ahess$.
Although we only have $\wt \bY$ at our disposal, and not $\bA_+$, this is all that is needed; the procedures in Section~2 do not need to know $\ahess$ explicitly.
In addition, Proposition~\ref{prop:ls} is still valid, after replacing $\bY$ with $\wt\bY$.  We remark that, for our purpose, just a symmetric $\wt \bY^T\bS$ may also be sufficient, since we usually discard negative Ritz values.

In Section~\ref{sec:lmsd-exp} we test one possible way of computing $\Delta \bY$, as proposed in \cite{schnabel1983quasi}, and the Ritz values of the associated low-dimensional representation of $\ahess$. This application is new in the context of LMSD.
The perturbation is constructed as follows: first, consider the strict negative lower triangle $\bL$ of $\bY^T\bS - \bS^T\bY = -\bL + \bL^T$, and suppose $\bS$ is of full rank. (If not, remove the oldest $\bs$-vectors until the condition is satisfied.) Then $\bY^T\bS + \bL$ is symmetric.
Schnabel \cite{schnabel1983quasi} solves the underdetermined system $\Delta \bY^T \bS = \bL$, which has $\Delta \bY = \bS(\bS^T\bS)^{-1}\bL^T$ as minimum norm solution.
By Theorem~\ref{thm:sch}, there exists a symmetric $\ahess$ such that $\wt \bY =\ahess\, \bS$. Now let us consider the QR decomposition of $\bG = \bQ\bR$, which is of full rank since $\bS$ is also of full rank.
Similar to \eqref{eq:yg} we know that $\ahess \bG = -\wt \bY \bD$.
Moreover, we recall that $\bS = -\bG \bD^{-1}$, and that $\bY \bD = -\extmat{\bG}{\bg_{m+1}} \, \bJ$ from \eqref{eq:yg}.
For any perturbation $\Delta \bY$ that symmetrizes $\bY^T\bS$, we obtain the following low-dimensional representation of $\ahess$:
\begin{align}
\bQ^T\ahess\bQ &= \bQ^T\ahess \bG\bR^{-1} = -\bQ^T\wt \bY \bD\bR^{-1} = -\bQ^T(\bY+\Delta \bY) \, \bD\bR^{-1} \nonumber\\
&= \bT -\bQ^T\Delta \bY \, \bD\bR^{-1}.\label{eq:lmsd-pert0}
\end{align}
where $\bT = \extmat{\bR}{\bfr}\, \bJ\bR^{-1}$ as in \eqref{eq:fletcherT}.
In particular, by replacing $\Delta \bY = \bS(\bS^T\bS)^{-1}\bL^T$, we obtain
\begin{equation}
    \bQ^T\ahess\bQ = \bT + \bR(\bR^T\bR)^{-1}\bD\bL^T\bD\bR^{-1}.\label{eq:lmsd-pert}
\end{equation}
This means that \eqref{eq:lmsd-pert} can be computed by means of the Cholesky decomposition of $\bG^T\bG$ only; the factor $\bQ$ is not needed. 

We now give a new interpretation of Fletcher's extension of LMSD to general nonlinear problems \cite[Sec.~4]{fletcher2012limited}, in terms of a specific perturbation of $\bY$. Fletcher notices that the matrix $\bT$ \eqref{eq:fletcherT} is an upper Hessenberg matrix and can still be computed from the matrix of gradients, but, because of Theorem~\ref{thm:sch}, there is no guarantee that $\bT$ corresponds to a low-dimensional representation of a symmetric approximate Hessian matrix. Since the eigenvalues of $\bT$ might be complex, Fletcher proposes to enforce $\bT$ to be tridiagonal by replacing its strict upper triangular part with the transpose of its strict lower triangular part.
We now show that this operation in fact corresponds to a perturbation of the $\bY$ matrix. To the best of our knowledge, this result is new. 
\begin{proposition}  \label{prop:fletcher}
Let $\bT$ be as in \eqref{eq:fletcherT} and consider its decomposition $\bT = \bL + \bLambda + \bU$, where $\bL$ ($\bU$) is strictly lower (upper) triangular and $\bLambda$ is diagonal. Moreover, let $\bG$ be full rank and $\bG = \bQ\bR$ its QR decomposition. If 
 \[
 \Delta \bY = \bQ\,(\bU-\bL^T)\, \bR\bD^{-1},
 \]
 then $\wt \bY^T\bS = (\bY + \Delta \bY)^T\bS$ is symmetric and there exists a symmetric $\ahess$ such that $\wt \bY = \ahess\bS$ and $\bQ^T\ahess\bQ = \bL + \bLambda + \bL^T$.
\end{proposition}
\begin{proof}
 First, we prove that $\bS^T\wt \bY$ is symmetric. By replacing the expression for $\Delta \bY$ and exploiting the QR decomposition of $\bG$, we get
 \begin{align*}
 \bS^T\wt \bY &= -\bD^{-1}\bR^T\big(-\extmat{\bR}{\bfr}\, \bJ\bR^{-1} + \bU - \bL^T\big)\, \bR\bD^{-1}\\
 &= -\bD^{-1}\bR^T\big(-(\bL + \bLambda + \bU) + \bU - \bL^T\big)\, \bR\bD^{-1}\\
 &= \bD^{-1}\bR^T\big(\bL + \bLambda + \bL^T\big)\, \bR\bD^{-1}.
 \end{align*}
 Therefore $\bS^T\wt \bY$ is symmetric; Theorem~\ref{thm:sch} implies that there exists a symmetric $\ahess$ such that $\wt \bY = \ahess\bS$. From this secant equation, it follows from \eqref{eq:lmsd-pert0} that
 \begin{align*}
\bQ^T\ahess\bQ = \big(\bL + \bLambda + \bL^T\big)\, \bR\bD^{-1}(\bD\bR^{-1}) = \bL + \bLambda + \bL^T.
 \end{align*}
\end{proof}
From this proposition, we are able to provide an upper bound for the spectral norm of the perturbation $\Delta \bY$:
\[
\|\Delta \bY\|_2 \le \max_{i}\beta_i\cdot\|\bR\|_2 \cdot \|\bT- (\bL + \bLambda + \bL^T)\|_2 ,
\]
where $\max_{i}\beta_i$ is the largest stepsize among the latest $m$ steps. This suggests that the size of the perturbation $\Delta \bY$ is determined not only by the distance between $\bT$ and its symmetrization, as expected, but also by the spectral norm of $\bR$: if this is large, the upper bound may also be large. 

We would like to point out the following important open and intriguing question. While Schnabel solves $\Delta \bY^T \bS = \bL$ to symmetrize $\bY^T \bS$, and Fletcher's update is described in Proposition~\ref{prop:fletcher}, there may be other choices for the perturbation matrix $\Delta \bY$ that, e.g., have a smaller $\Delta \bY$ in a certain norm. However, obtaining these perturbations might be computationally demanding, compared to the task of getting $m$ new stepsizes. In the cases we have analyzed, the lower-dimensional $\ahess$ can be obtained from the Cholesky decomposition of $\bG^T\bG$ at negligible cost.

Given the generality of Schnabel's Theorem~\ref{thm:sch}, another possibility that may be explored is a perturbation of $\bS$, rather than $\bY$, to symmetrize $\bS^T\bY$. This would be a natural choice for computing the harmonic Ritz values given a basis for $\bY$. In this situation, the matrix binding $\bS$ and $\bY$ would play the role of an approximate inverse Hessian. A thorough investigation is out of the scope of this paper.

\subsection{Symmetric solutions to the secant equations} \label{sec:symsec}
In this subsection, we explore a second and alternative extension of LMSD.
We start from the secant condition of Proposition~\ref{prop:ls} for a low-dimensional matrix $\bB$.
The key idea is to impose \emph{symmetry constraints} to obtain real eigenvalues from the solutions to the least squares problems of Proposition~\ref{prop:ls}. Even if the hypothesis of Theorem~\ref{thm:sch} is not met, this method still fulfills the purpose of obtaining new stepsizes for the LMSD iterations.

The following proposition gives the stationarity conditions to solve the two modified least squares problems. Denote the symmetric part of a matrix by $\sym(\bA) := \frac12(\bA + \bA^T)$.

\begin{proposition} \label{prop:ls-sym}
Let $\bS$, $\bY\in \mathbb{R}^{n\times m}$ be full rank, with $n \ge m$, and $\bB$, $\bH\in \mathbb{R}^{m\times m}$.
\begin{enumerate}
\item[(i)] The solution to $\ds \min_{\bB=\bB^T} \|\bY-\bS\bB\|$ satisfies $\sym(\bS^T\bS\, \bB - \bS^T\bY) = {\bf0}$.
\item[(ii)] The solution to $\ds \min_{\bH=\bH^T} \|\bY\bH-\bS\|$ satisfies $\sym(\bY^T\bY\bH - \bY^T\bS) = {\bf0}$.
\end{enumerate}
\end{proposition}
\begin{proof}
If $\bB$ is symmetric, it holds that
\[
\|\bY-\bS\bB\|^2 = \tr(\bB\, \bS^T\bS\, \bB - 2\, \sym(\bS^T\bY)\, \bB + \bY^T\bY),
\]
where $\tr(\cdot)$ denotes the trace of a matrix. Differentiation leads to the following stationarity condition for $\bB$:
\begin{equation} \label{eq:lya}
\bS^T\bS\, \bB + \bB\, \bS^T\bS = 2\, \sym(\bS^T\bY),
\end{equation}
which is a Lyapunov equation. Since $\bS$ is of full rank, its Gramian matrix is positive definite.
This implies that the spectra of $\bS^T\bS$ and $-\bS^T\bS$ are disjoint, and therefore the equation admits a unique solution
(see, e.g., \cite{simoncini2016computational} for a review of the Lyapunov equation and properties of its solution). Part (ii) follows similarly. 
\end{proof}

It is easy to check that, for $m=1$, $\bB$ in part (i) reduces to $\alpha^{\rm BB1}$ and $\bH$ in part (ii) to $\beta^{\rm BB2}$.
Compared to Fletcher's $\bT$ matrix \eqref{eq:fletcherT}, the symmetric solutions $\bB$ and $\bH$ will generally give a larger residual (since they are suboptimal for the unconstrained secant conditions), but they enjoy the benefit that their eigenvalues are guaranteed to be real.

We remark that symmetry constraints also appear in the secant conditions of the BFGS method, and in the symmetric rank-one update (see, e.g., \cite[Chapter~6]{nocedal2006numerical}).
While in the BFGS method the approximate Hessians are SPD, provided that the initial approximation is SPD, in the rank-one update method it is possible to get negative eigenvalues. The fundamental difference between LMSD and these methods is that we do not attempt to find an approximate $n\times n$ Hessian matrix.

Even while we do not approximate the eigenvalues of some Hessian, as in the quadratic case of Section~\ref{sec:eig}, it is possible to establish bounds for the extreme eigenvalues of the solutions to the Lyapunov equations of Proposition~\ref{prop:ls-sym}, provided that $\sym(\bS^T\bY)$ is positive definite.
The following result is a direct consequence of \cite[Cor.~1]{yasuda1979upper}.

\begin{proposition}
Given the solution $\bB$ to \eqref{eq:lya}, let $\lambda_1(\bB)$ ($\lambda_n(\bB)$) be the smallest (largest) eigenvalue of $\bB$. If $\bS$ is of full rank and $\sym(\bS^T\bY)$ is positive definite, then 
\[
[\lambda_1(\bB), \, \lambda_n(\bB)] \subseteq
[\lambda_1((\bS^T\bS)^{-1}\sym(\bS^T\bY)),\ \lambda_n((\bS^T\bS)^{-1}\sym(\bS^T\bY))].
\]
If there exists an SPD matrix $\ahess$ such that $\bY = \ahess\bS$, then $[\lambda_1(\bB), \, \lambda_n(\bB)] \subseteq [\lambda_1(\ahess), \, \lambda_n(\ahess)]$.
\end{proposition}
\begin{proof}
The first statement directly follows from \cite[Cor.~1]{yasuda1979upper}. From this we have
\[
\lambda_1(\bB) \ge -\lambda^{-1}_1(-\bS^T\bS \,(\sym(\bS^T\bY))^{-1}),\quad \lambda_n(\bB) \le -\lambda^{-1}_n(-\bS^T\bS \, (\sym(\bS^T\bY))^{-1}).
\]
The thesis follows from the fact that, given a nonsingular matrix $\bLambda$ with positive eigenvalues, the following equality holds for the largest and the smallest eigenvalue: $\lambda_i(-\bLambda^{-1}) = -1/\lambda_i(\bLambda)$, where $i =1,\,n$.

When $\bY = \ahess\bS$, from Cauchy's Interlace Theorem, the spectrum of $\bB$ lies in $[\lambda_1(\ahess), \lambda_n(\ahess)]$.
\end{proof}

An analogous result can be provided for the matrix $\bH$ of Proposition~\ref{prop:ls-sym}(ii).

\subsection{Solving the Lyapunov equation while handling rank deficiency}
\label{sec:lya-ld}
The solution to the Lyapunov equation \eqref{eq:lya} is unique, provided that $\bS$ is of full rank.
In this section, we propose three options if $\bS$ is (close to) rank deficient. As in Section~\ref{sec:eig}, we discuss approaches using a Cholesky decomposition, a pivoted QR factorization, and a truncated SVD.
By using the decompositions exploited in Section~\ref{sec:ritz} we can either discard some $\bs$-vectors (and their corresponding $\by$-vectors) or solve the Lyapunov equation onto the space spanned by the first right singular vectors of $\bS$.

In the Cholesky decomposition and the pivoted QR decomposition we remove some of the $\bs$-vectors and the corresponding $\by$-vectors, if needed. As we have seen in Section~\ref{sec:ritz}, in the Cholesky decomposition we discard past gradients until the Cholesky factor $\bR$ of $\bG^T\bG$ is sufficiently far from singular.
In this new context of Lyapunov equations, we additionally need the relation (cf.~\eqref{eq:yg})
\begin{equation}
\label{eq:yk}
\bY = \extmat{\bG}{\bg_{m+1}}\, \bK, \qquad \text{where} \quad
\bK = 
\begin{bmatrix}
-1 \\[-2mm]
\ph{-}1& \ddots \\[-1.5mm]
& \ddots & -1 \\
&& \ph{-}1\\
\end{bmatrix}.
\end{equation}
Then we can easily compute the matrices present in \eqref{eq:lya}:
\begin{equation}
\label{eq:lmsd-lya}
\bS^T\bS = \bD^{-1}\bR^T\bR\bD^{-1},\qquad
\bS^T\bY = -\bD^{-1}\bR^T\extmat{\bR}{\bfr}\, \bK.
\end{equation}
In the pivoted QR, we keep only the columns $\bG\wh\bPi_G$ and $\bY\wh\bPi_G$, as in Section~\ref{sec:eig}. 
Let $\bD_G$ be the diagonal matrix that stores the inverse stepsizes corresponding to $\bG\wh\bPi_G$. Then
\begin{align*}
\wh\bPi_G^T \, \bS^T\bS \, \wh\bPi_G &= \bD_G^{-1} \, \wh\bR_G^T\wh\bR_G \, \bD_G^{-1}, \\
\wh\bPi_G^T \, \bS^T\bY \, \wh\bPi_G &= -\bD_G^{-1}\extmat{\wh\bR_G^T \, [\wh\bR_G \ \wh\bR_{12}]\, \wh\bPi^{-1}}{\wh\bPi_G^T\bG^T\bg_{m+1}}\, \bK\wh\bPi_G,
\end{align*}
 where $\wh\bR_G$ and $\wh\bR_{12}$ come from the block representation of $\wh\bR$ \eqref{eq:Rpivo}.

The third approach involves the SVD of $\bS$, instead of the SVD of $\bG$ as in Section~\ref{sec:ritz}. This is due to the fact that the solution to the Lyapunov equation \eqref{eq:lya} for $\bS$ and $\bY$ is not directly related to the solution to the Lyapunov equation for $\bS\bLambda$ and $\bY\bLambda$, for a nonsingular $\bLambda$. From the SVD $\bS = \wh\bU\wh\bSigma \wh\bV^T$, we get
\[
\wh\bV\wh\bSigma^2 \wh\bV^T\, \bB + \bB \, \wh\bV\wh\bSigma^2 \wh\bV^T = 2\, \sym(\bS^T\bY).
\] 
To simplify this equation, consider the truncated SVD $\bS_1 = \bU_S\bSigma_S\bV_S^T$, where $\bSigma_S$ is $s\times s$, and multiply by $\bV_S^T$ on the left and by $\bV_S$ on the right. Since $\bV_S^T\wh\bV\bSigma^2 \wh\bV^T = \bSigma_S^2\bV_S^T$,
\begin{equation}\label{eq:lya-svd}
    \bSigma_S^2 \, \bB_S + \bB_S \, \bSigma_S^2 = 2\, \sym(\bSigma_S \bU_S^T \bY \bV_S),
\end{equation}
where $\bB_S = \bV_S^T\bB\bV_S$ is the projection of $\bB$ onto $\bV_S$. Moreover, from \eqref{eq:yk} we get
\[
\bU_S^T\bY = \extmat{-\bSigma_S \bV_S^T\bD}{\bU_S^T\bg_{m+1}}\, \bK.
\]
We remark that it is appropriate to control the truncation of the SVD by the condition number of the coefficient matrix $\bS^T\bS$, which is $\kappa^2(\bS)$. 

The previous discussion on the three decompositions can also be extended to the secant equation of Proposition~\ref{prop:ls-sym}(ii), to compute the matrix $\bH$ and use its eigenvalues directly as stepsizes.
Several possibilities may be explored by decomposing either $\bG$ or $\bY$ as in Section~\ref{sec:harm} for the harmonic Ritz values. We will not discuss any further details regarding all these methods, but in the experiments in Section~\ref{sec:lmsd-exp} we will present results obtained with the Cholesky factorization of $\extmat{\bG}{\bg_{m+1}}^T\extmat{\bG}{\bg_{m+1}}$ as expressed in \eqref{eq:fletcherP}. 
Then for the quantities in Proposition~\ref{prop:ls-sym}(ii) we have:
\[
\bY^T\bY = \bK^T\begin{bmatrix}
\bR & \bfr \\ {\bf0} & \rho
\end{bmatrix}^T 
\begin{bmatrix}
\bR & \bfr \\ {\bf0} & \rho 
\end{bmatrix}\bK,
\]
and the matrix $\bY^T\bS$ can be obtained from \eqref{eq:lmsd-lya}.

We note that all Lyapunov equations in this section are of the form $\bE^T\bE \bB + \bB \bE^T\bE = \bF$. We describe a practical solution approach. Consider the truncated SVD $\bE \approx \bU_E\bSigma_E\bV_E^T$, where the singular values in $\bSigma_E$ satisfy $\sigma_i^2(\bE) \ge {\sf thresh} \cdot\sigma_1^2(\bE)$. 
In case we exploit the Cholesky decomposition or the pivoted QR, an extra truncated SVD might still be appropriate, since these two decompositions do not provide an accurate estimate of $\kappa^2(\bE)$. By left and right multiplication by $\bV_E$, we obtain an expression analogous to \eqref{eq:lmsd-lya}:
\[
\bSigma_E^2 \, \bB_E + \bB_E \, \bSigma_E^2 = \bV_E^T\bF\bV_E,
\]
where $\bB_E = \bV_E^T\bB\bV_E$. Since $\bSigma_E$ is diagonal, the solution to this Lyapunov equation can be easily found by elementwise division (cf.~\cite[p.~388]{simoncini2016computational}):
\[
[\bB_E]_{ij} = [\bV_E^T\bF\bV_E]_{ij} \ / \ (\sigma^2_i(\bE) + \sigma^2_j(\bE)).
\]
We notice that, in the SVD approach \eqref{eq:lmsd-lya}, the solution can be found directly from this last step. In addition, we remark that, for the scope of LMSD, it is not necessary to find the solution $\bB$ to the original Lyapunov equation.

\subsection{An algorithm for general nonlinear functions} 
\label{sec:algo-general}
We now review the limited memory steepest descent for general unconstrained optimization problems, as implemented in~\cite[Alg.~2]{daniela2018steplength} and reported in Algorithm~\ref{algo:lmsd-unconstrained}. Compared to the gradient method for strictly convex quadratic functions, LMSD for general nonlinear functions has more complications. 
In Section~\ref{sec:genf} we have proposed two alternative ways to find a set of real eigenvalues to use as stepsizes. However, we may still get negative eigenvalues. This problem also occurs in classical gradient methods, when $\bs_k^T\by_k < 0$: in this case, the standard approach is to replace any negative stepsize with a positive one. 
In LMSD, we keep $s \le m$ positive eigenvalues and discard the negative ones. If all eigenvalues are negative, we restart from $\beta_k = \max(\min(\|\bg_k\|_2^{-1}, \, 10^{5}), \ 1)$ as in \cite{raydan1997barzilai}. Moreover, as in \cite{daniela2018steplength}, only the latest $s$ gradients are kept. As an alternative to this strategy, we also mention the more elaborated approach of Curtis and Guo \cite{curtis2016handling}, which involves the simultaneous computation of Ritz and harmonic Ritz values.

The line search of LMSD in \cite{daniela2018steplength} is inspired by Algorithm~\ref{algo:lmsd} for quadratic functions, and has many similarities with the routine proposed by Fletcher \cite{fletcher2012limited}. Once new stepsizes have been computed, at each sweep we produce a new iterate starting from the smallest stepsize in the stack. The reference function value $f_{\rm ref}$ for the Armijo sufficient decrease condition is the function value at the beginning of the sweep, as in Algorithm~\ref{algo:lmsd}. 
We note that this Armijo type of line search appropriately replaces the exact line search of Algorithm~\ref{algo:lmsd}, i.e., the choice of the Cauchy stepsize when a nonmonotone behavior (with respect to $f_{\rm ref}$) is observed.
The stack of stepsizes is cleared whenever the current steplength needs to be reduced to meet the sufficient decrease condition, or when the new gradient norm is larger than the previous one. This requirement is also present in Algorithm~\ref{algo:lmsd}. Notice that, since we terminate the sweep whenever a backtracking step is performed, starting from the smallest stepsizes decreases the likelihood of ending a sweep prematurely. 
In contrast with \cite{daniela2018steplength}, we keep storing the past gradients even after clearing the stack. This choice turns out to be favorable for the experiments in Section~\ref{sec:lmsd-expunc}.

\begin{algorithm}[htb!]
\caption{LMSD for general nonlinear functions \cite{daniela2018steplength}}
{\bf Input}: Continuously differentiable function $f$, initial guess $\bx_0$, initial stepsize $\beta_0>0$, tolerance {\sf tol}; safeguarding parameters $\beta_{\max} > \beta_{\min} > 0$; line search parameters $c_{\rm{ls}}$, $\sigma_{\rm{ls}} \in (0,1)$ \\
{\bf Output}: Approximation to minimizer $\argmin_{\bx} f(\bx)$ \\
\begin{tabular}{ll}
{\footnotesize 1}: & $\bg_0 = \nabla f(\bx_0)$, \quad $f_{\rm ref} = f(\bx_0)$ \\
{\footnotesize 2}: & $j = 0$,\quad $s = 1$ \hfill {\# $s$ is the stack size}\\
{\footnotesize 3}: & {\bf for} $k = 0, 1, \dots$ \\
{\footnotesize 4}: & \ph{M} $\nu_k = \max(\beta_{\min}, \ \min(\beta_j, \, \beta_{\max}))$,\quad $j = j+1$ \\
{\footnotesize 5}: & \ph{M} {\bf if} $f(\bx_k -\nu_k \bg_k) \le f_{\rm ref} - c_{\rm{ls}} \, \nu_k \, \|\bg_k\|^2$ \\
{\footnotesize 6}: & \ph{MM} $\bx_{k+1} = \bx_k -\nu_k \bg_k$ \\
{\footnotesize 7}: & \ph{M} {\bf else} \\
{\footnotesize 8}: & \phantom{MM} {\bf while} \ $f(\bx_k-\nu_k \, \bg_k) > f_{\text{ref}} - c_{\rm{ls}} \, \nu_k \, \|\bg_k\|^2$ \ {\bf do} \ $\nu_k = \sigma_{\rm{ls}} \, \nu_k$ \ {\bf end} \\
{\footnotesize 9}: & \ph{MM} $\bx_{k+1} = \bx_k -\nu_k \bg_k$ \\
{\footnotesize 10}: & \ph{MM} clear the stack \\
{\footnotesize 11}: & \ph{M} {\bf end} \\
{\footnotesize 12}: & \ph{M} {\bf if} \ $\|\bg_{k+1}\| \le {\sf tol} \cdot \|\bg_0\|$, \ {\bf return}, \ {\bf end} \\
{\footnotesize 13}: & \ph{M} {\bf if} \ $\|\bg_{k+1} \| \ge \|\bg_k \|$, clear the stack \ {\bf end} \\
{\footnotesize 14}: & \ph{M} {\bf if} empty stack {\bf or} $j\ge s$\\
{\footnotesize 15}: & \ph{MM} Compute stack of $s\le m$ new stepsizes $\beta_j > 0$, ordered increasingly \\
{\footnotesize 16}: & \ph{MM} Store only last $s$ most recent vectors of $\bG$
\\
{\footnotesize 17}: & \ph{MM} $j=0$,\quad $f_{\rm ref} = f(\bx_{k+1})$\\
{\footnotesize 18}: & \ph{M} {\bf end} \\
{\footnotesize 19}: & {\bf end} \\
\end{tabular}
\label{algo:lmsd-unconstrained}
\end{algorithm}

We remark that, by construction, all new function values within a sweep are smaller than $f_{\rm ref}$. Therefore, the line search strategy adopted in \cite{daniela2018steplength} can be seen as a nonmonotone line search strategy \cite{grippo1986nonmonotone}. Given the uniform bounds imposed on the sequence of stepsizes, the result of global convergence for a gradient method with nonmonotone line search \cite[Thm.~2.1]{raydan1997barzilai} also holds for Algorithm~\ref{algo:lmsd-unconstrained}. 

\section{Numerical experiments} \label{sec:lmsd-exp}
We explore the several variants of LMSD, for the quadratic and for the general unconstrained case. We compare LMSD with two gradient methods which adaptively pick either BB1 or BB2 stepsizes \cite{frassoldati2008new, bonettini2008scaled}. As claimed by Fletcher \cite{fletcher2012limited}, we have observed that LMSD may indeed perform better than L-BFGS on some problems. 
However, in the majority of our test cases, L-BFGS, as implemented in \cite{byrd1995limited}, 
converges faster than LMSD, in terms of number of function (and gradient) evaluations, and computational time. The comparison with another gradient method seems fairer to us than the comparison with a second-order method, and therefore we will not show L-BFGS in our study. Nevertheless, as discussed in Section~\ref{sec:lmsd-intro}, we recall the two main advantages of considering LMSD methods: the possibility to extend its idea to problems beyond unconstrained optimization (see, e.g., \cite{porta2015,franchini2020ritz,crisci2023hybrid}), and the less stringent requirements on the objective functions to guarantee the global convergence of the method.


\subsection{Quadratic functions}
\label{sec:lmsd-expquad}

The performance of the LMSD method may depend on several choices: the memory parameter $m$, whether we compute Ritz or harmonic Ritz values, and how we compute a basis for either $\cals$ or $\caly$. This section studies how different choices affect the behavior of LMSD in the context of strictly convex quadratic problems \eqref{quad}.

We consider quadratic problems by taking the Hessian matrices from the SuiteSparse Matrix Collection \cite{suitesparse}. 
These are 103 SPD matrices with a number of rows $n$ between $10^2$ and $10^4$. From this collection we exclude only {\sf mhd1280b}, {\sf nd3k}, {\sf nos7}. The vector $\bb$ is chosen so that the solution of $\bA\bx = \bb$ is $\bx^\ast = \be$, the vector of all ones.
For all problems, the starting vector is $\bx_0 = 10\, \be$, and the initial stepsize is $\step{}{0} = 1$. The algorithm stops when $\|\bg_k\| \le {\sf tol} \cdot \|\bg_0\|$ with ${\sf tol} = 10^{-6}$, or when $5\cdot 10^4$ iterations are reached. We compare the performance of LMSD with memory parameters $m = 3,\,5,\, 10$ with the $\abbmin$ gradient method \cite{frassoldati2008new} and one of its variants, presented in \cite{bonettini2008scaled}, which we indicate with $\abbbon$. The $\abbmin$ stepsize is defined as
\[
\step{ABB_{\rm min}}k = \left\{
\begin{array}{ll}
\min\{\step{BB2}j\mid j = \max\{1, k-m\},\dots,k\}, & \ \text{if} \ \step{BB2}k < \eta \, \step{BB1}k, \\[1.5mm]
\step{BB1}k, & \ \text{otherwise,}
\end{array}
\right.
\]
where $m = 5$ and $\eta = 0.8$. The $\abbbon$ stepsize is defined in the same way as $\abbmin$, but with an adaptive threshold $\eta$: starting from $\eta_0 = 0.5$, this is updated as
\[
\eta_{k+1} = \left\{
\begin{array}{ll}
0.9\,\eta_k, & \ \text{if} \ \step{BB2}k < \eta_k \, \step{BB1}k, \\[0.5mm]
1.1\,\eta_k, & \ \text{otherwise.}
\end{array}
\right.
\]

Since the performance of these methods depends less on the choice of $m$ than LMSD, we only show $m=5$ for both $\abbmin$ and $\abbbon$. Among many possible stepsize choices, we compare LMSD with these gradient methods because they behave better than classical BB stepsizes on quadratic problems (see, e.g., \cite{FHK23}).

We recall that one $\abbmin$ ($\abbbon$) step requires the computation of three inner products of cost $\calo(n)$ each. An LMSD sweep is slightly more expensive, involving operations of order $m^2n$ and (much less important) $m^3$, but it is performed approximately once every $m$ iterations.
These costs correspond to the decomposition of either $\bG$ or $\bY$, the computation of the projected Hessian matrices and their eigenvalues. We also remark that, while pivoted QR and SVD require $\calo(m^2n)$ operations, the Cholesky decomposition is $\calo(m^3)$, but is preceded by the computation of a Gramian matrix, with cost $\calo(m^2n)$.

We consider two performance metrics: the number of gradient evaluations (NGE) and the computational time. The number of gradient evaluations also includes the iterations that had to be restarted with a Cauchy step (cf.~Algorithm~\ref{algo:lmsd}, Line~9).
Our experience indicates that computational time may depend significantly on the chosen programming language, and therefore should not be the primary choice in the comparison of the methods.
Nevertheless, it is included as an indication, because it takes into account the different costs of an LMSD sweep and $m$ iterations of a gradient method.

The comparison of different methods is made by means of the performance profile \cite{perfprofile}, as it is implemented in Python's library \href{https://pypi.org/project/perfprof/}{{\sf perfprof}}. 
Briefly speaking, the cost of each algorithm per problem is normalized, so that the winning algorithm has cost $1$. This quantity is called \emph{performance ratio}. Then we plot the proportion of problems that have been solved within a certain performance ratio. An infinite cost is assigned whenever a method is not able to solve a problem to the tolerance within the maximum number of iterations.  

We compare the performance of LMSD where the stepsizes are computed as summarized in Table~\ref{tab:lmsd-exp-quad}. 
\begin{table}[htb!]
\centering
\caption{Strategies to compute the new stack of stepsizes in LMSD methods for quadratic functions. RQ refers to the computation of Rayleigh quotients from the harmonic Ritz vectors. H stands for ``harmonic'', the letters G (Y) indicate whether a decomposition has been used to implicitly compute a basis for $\bG$ ($\bY$).}
{\footnotesize
\begin{tabularx}{\textwidth}{lXl}
\toprule
Method & Description & Matrix \\
\midrule
LMSD-G \cite{fletcher2012limited} & Cholesky on $\bG^T \bG$ to compute the inverse Ritz values of $\bA$ & $\bT$ \eqref{eq:fletcherT} \\
LMSD-G-QR & Pivoted QR on $\bG$ to compute the inverse Ritz values of $\bA$ & $\bB^{\rm QR}$ \eqref{eq:qrbb1} \\
LMSD-G-SVD & SVD on $\bG$ to compute the inverse Ritz values of $\bA$ & $\bB^{\rm SVD}$ \eqref{eq:svdbb1-g} \\
LMSD-HG \cite{fletcher2012limited} & $\bA\cdot{\rm span}(\bG)$ to compute the inverse harmonic Ritz values of $\bA$ & $\bT^{-1}\bP$ \eqref{eq:harmfletcher} \\
LMSD-HG-RQ & $\bA\cdot{\rm span}(\bG)$ to find the harmonic Ritz vectors $\bA$ and next compute their inverse Rayleigh quotients (cf.~end of Sec.~\ref{sec:harm}) & $\bT^{-1}\bP$ \eqref{eq:harmfletcher} \\
LMSD-HY & Cholesky on $\bY^T \bY$ to compute the Ritz values of $\bA^{-1}$ & $\bH^{\rm CH}$ \eqref{eq:HYR} \\
 \bottomrule
\end{tabularx}
}
\label{tab:lmsd-exp-quad}
\end{table}
In the first comparison we only consider methods that involve a Cholesky decomposition for simplicity. The Cholesky routine raises an error any time the input matrix is not SPD; therefore no tolerance {\sf thresh} for discarding old gradients needs to be chosen. The performance profiles for this first experiment are shown in Figure~\ref{fig:bbcomparison} for $m \in\{3,5,10\}$, in the performance range $[1,\,3]$. As $m$ increases all methods improve, both in terms of gradient evaluations and computational time. 

\begin{figure}[hb!]
\centering
\includegraphics[width=\textwidth]{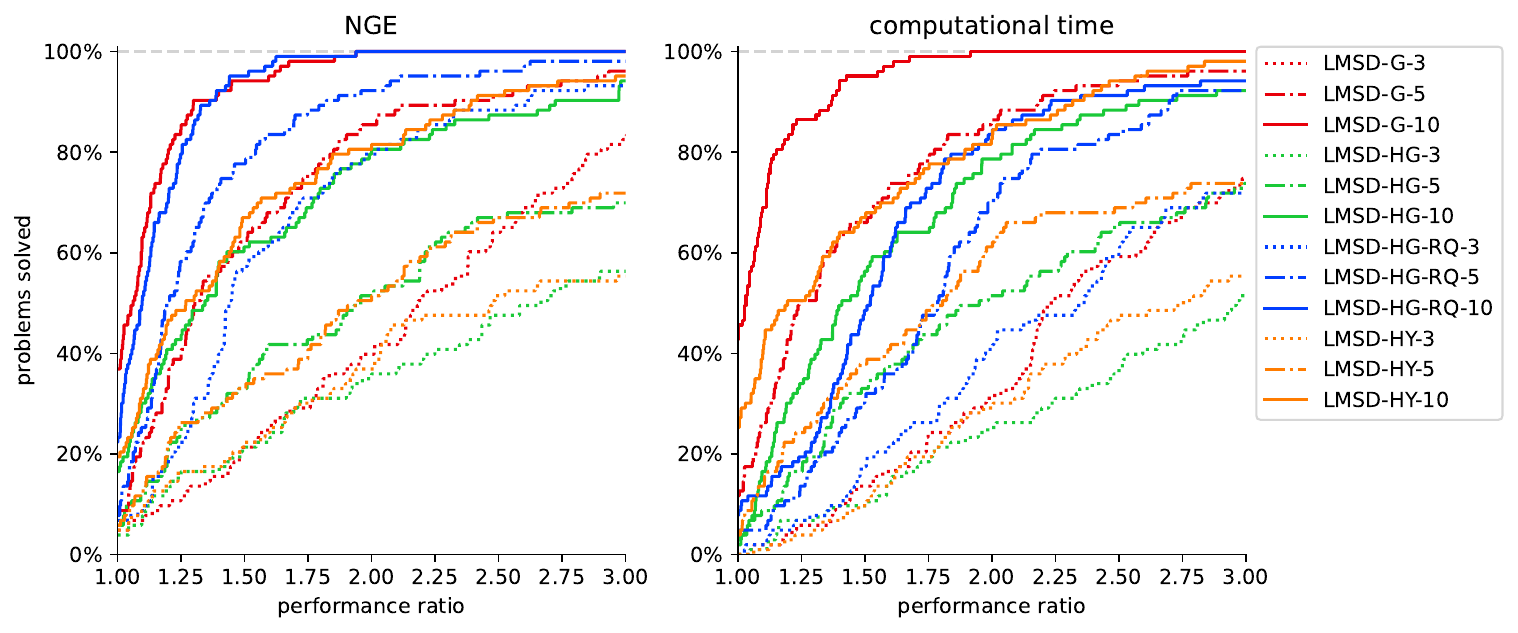}
\caption{Performance profile for strictly convex quadratic problems, based on the number of gradient evaluations (left) and computational time (right). Different line types indicate different values for $m$. Comparison between the computation of Ritz values or harmonic Ritz values.}
\label{fig:bbcomparison}
\end{figure}

The method that performs best, both in terms of NGE and computational time, is LMSD-G for $m = 10$. When $m = 5$, LMSD-HG-RQ performs better than LMSD-G in terms of NGE, but it is more computationally demanding. This is reasonable, since LMSD-HG-RQ has to compute $m$ extra Rayleigh quotients; this operation has an additional cost of $m^3$, which can be relatively large for some problems in our collection, where $m^3 \approx n$. 

LMSD-HG and LMSD-HY perform similarly, since they are two different ways of computing the same harmonic Ritz values. They generally perform worse than the other two methods; in the case $m = 10$, their performances are comparable with those of LMSD-G for $m = 5$.

\begin{figure}[htb!]
\centering
\includegraphics[width=\textwidth]{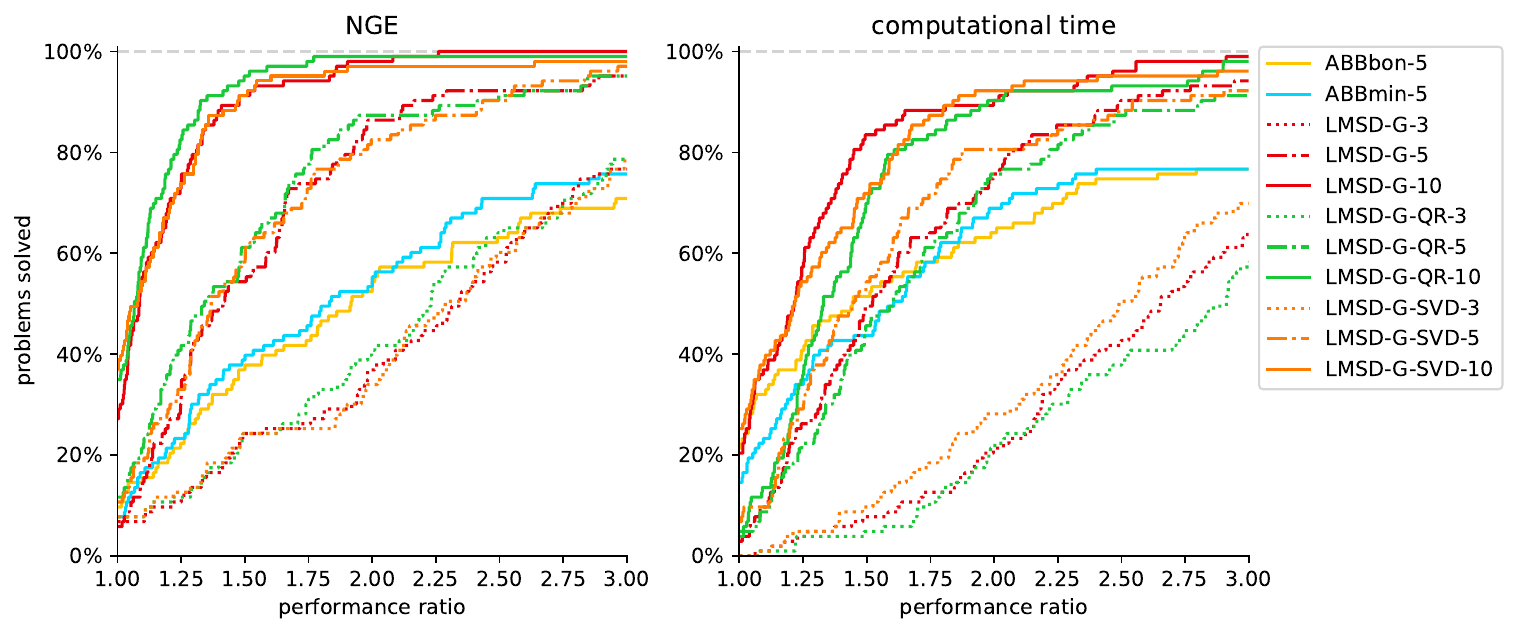}
\caption{Performance profile for strictly convex quadratic problems, based on the number of gradient evaluations (left) and computational time (right). Different line types indicate different values for $m$. Comparison between different decompositions for the matrix $\bG$.}
\label{fig:basescomparison}
\end{figure}

In Figure~\ref{fig:basescomparison} we compare LMSD with $\abbmin$ and $\abbbon$. Given the comments to Figure~\ref{fig:bbcomparison}, we decide to compute the Ritz values of the Hessian matrix, by decomposing $\bG$ in different ways. Specifically, we compare LMSD-G, LMSD-G-QR and LMSD-G-SVD (cf.~Table~\ref{tab:lmsd-exp-quad}). The tolerance to decide the memory size $s \le m$ is set to ${\sf thresh} = 10^{-8}$, for both pivoted QR and SVD. Once more, we clearly see that LMSD improves as the memory parameter increases, both in terms of gradient evaluations and computational time. Once $m$ is fixed, the three methods to compute the basis for $\cals$ are almost equivalent. LMSD-G-SVD seems to be slightly faster than LMSD-G in terms of computational time, as long as the performance ratio is smaller than $1.5$. In our implementation, LMSD-G-QR seems to be more expensive. Compared to $\abbmin$ and $\abbbon$, all LMSD methods with $m = 5,10$ perform better in terms of gradient evaluations. LMSD-G-SVD, for $m=10$, appears to be faster than $\abbmin$ and $\abbbon$, also in terms of computational time.

\begin{figure}[htb!]
\centering
\includegraphics[width=0.65\textwidth]{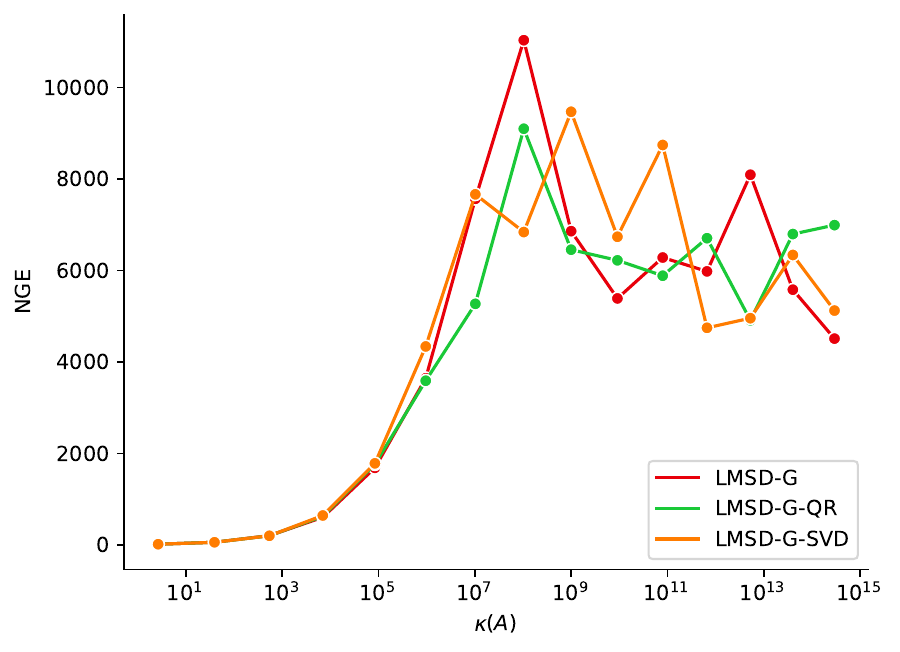}
\caption{Condition number of quadratic problems with Hessian matrix $\bA$ and corresponding number of gradient evaluations. Different colors indicate different ways of computing the Ritz values of $\bA$.}
\label{fig:lmsdqp2}
\end{figure}

Figure~\ref{fig:basescomparison} already suggests that different decompositions give approximately equivalent results. In addition, given a problem, it is difficult to recommend a certain decomposition strategy. We illustrate this idea with the following example: consider a family of $15$ problems with $\bA = \text{diag}(1,\omega,\omega^2, \dots,\omega^{99})$, where $\omega$ assumes $15$ values equally spaced in $[1.01,\,1.4]$. Geometric sequences as eigenvalues are quite frequent in the literature; see, e.g., \cite{fletcher2012limited,daniela2018steplength}. The starting vector is $\bx_0 = \be$, the associated linear system is $\bA\bx = {\bf 0}$; the memory parameter is $m = 5$, and each problem is scaled by the norm of the first gradient, so that ${\sf tol} = 10^{-7}/\|\bg_0\|$. The initial stepsize is $\beta_0 = 0.5$. In Figure~\ref{fig:lmsdqp2}, we plot the condition number of $\bA$ against the number of gradient evaluations. The three methods start to differ already with $\kappa(\bA) \approx 10^5$. For a large condition number, there is no clear winner in the performed experiments.

To summarize, when the objective function is strictly convex quadratic, Ritz values seem preferable over harmonic Ritz values. 
This is emphasized by the improvement of LMSD-HG when taking Rayleigh quotients instead of harmonic Rayleigh quotients. Different decompositions of $\bG$ result in mild differences in the performance of LMSD. Even if Cholesky decomposition is the least stable from 
a numerical point of view, its instability does not seem to have a clear effect on the performance of LMSD.
Finally, we observe that, in all methods, LMSD seems to improve as the memory parameter $m$ increases. 

\subsection{General unconstrained optimization problems}
\label{sec:lmsd-expunc}
In this section we want to assess the performance of LMSD for general unconstrained optimization problems, when we choose different methods to compute the stepsizes. These choices are summarized in Table~\ref{tab:lmsd-exp-unc}. 
\begin{table}[htb!]
\centering
\caption{Strategies to compute the new stack of stepsizes in LMSD methods for general nonlinear functions. H stands for ``harmonic''.}
{\footnotesize
\begin{tabularx}{\textwidth}{lX}
\toprule
Method & Description \\
\midrule
LMSD-CHOL \cite{fletcher2012limited} & Tridiagonalize $\bT$ as in \cite{fletcher2012limited} and compute its inverse eigenvalues \\
LMSD-H-CHOL \cite{curtis2016handling} & Symmetrize $\bP^{-1}\bT$ as in \cite{curtis2016handling} and compute its eigenvalues \\
LMSD-LYA & Inverse eigenvalues of the solution to Prop.~\ref{prop:ls-sym}\,(i) with Cholesky of $\bG^T\bG$ to handle rank deficiency \\
LMSD-LYA-QR & Idem with pivoted~QR of $\bG$ to handle rank deficiency \\
LMSD-LYA-SVD & Idem with SVD of $\bS$ to handle rank deficiency \\
LMSD-H-LYA & Eigenvalues of the solution to Prop.~\ref{prop:ls-sym}\,(ii) with Cholesky of $\extmat{\bG}{\bg_{m+1}}^T\extmat{\bG}{\bg_{m+1}}$ to handle rank deficiency \\
LMSD-PERT & Perturb $\bY$ according to \cite{schnabel1983quasi} to get \eqref{eq:lmsd-pert} and compute its inverse eigenvalues \\
 \bottomrule
\end{tabularx}
}
\label{tab:lmsd-exp-unc}
\end{table}
All the methods presented in Section~\ref{sec:genf} are considered, along with the extension of the harmonic Ritz values computation to the general unconstrained case. This is explained in \cite{curtis2016handling}, and indicated as LMSD-H-CHOL. In the quadratic case, the authors point out that the matrix $\bP$ \eqref{eq:fletcherP} can be expressed in terms of $\bT$ as $\bP = \bT^T \bT + \bxi\bxi^T$, where $\bxi^T = [{\bf 0 }^T\ \rho]\,\bJ\bR^{-1}$. Then, if $\wt \bT$ is the tridiagonal symmetrization of $\bT$ as in LMSD-CHOL, the new $\bP$ is defined as $\wt \bP = \wt \bT^T\wt \bT + \bxi\bxi^T$. The new stepsizes are the eigenvalues of $\wt \bP^{-1}\wt \bT$, and are real since $\wt \bP$ is generically SPD, and $\wt \bT$ is symmetric. 

All the LMSD methods are tested against the gradient method with nonmonotone line search \cite{raydan1997barzilai}. The stepsize choice is again $\abbmin$ with $m = 5$. The nonmonotone line search features a memory parameter $M = 10$; negative stepsizes are replaced by $\beta_k = \max(\min(\|\bg_k\|^{-1}, \, 10^{5}), \ 1)$, as in \cite{raydan1997barzilai}. In both algorithms, we set $\beta_{\min} = 10^{-30}$, $\beta_{\max} = 10^{30}$, $c_{\rm{ls}} = 10^{-4}$, $\sigma_{\rm{ls}} = \frac12$, and $\beta_{0} = \|\bg_0\|^{-1}$. The routine stops when $\|\bg_k\| \le {\sf tol}\cdot \|\bg_0\|$, with ${\sf tol} = 10^{-6}$, or when $10^5$ iterations are reached. In LMSD, the memory parameter has been set to $m \in \{3,5,7\}$.

We take 31 general differentiable functions from the CUTEst collection \cite{pycutest2022,cutest} and the suggested starting points $\bx_0$ therein. 
The problems are reported in Table~\ref{tab:uncpbs}. Since some test problems are non-convex, we checked whether all gradient methods converged to the same stationary point for different methods. As the performance profile, we may consider three different costs: the number of function evaluations (NFE), the number of iterations, and the computational time. The number of iterations coincides with the number of gradient evaluations for LMSD, $\abbmin$ and $\abbbon$.  
\begin{table}[htb!]
\centering
\caption{Problems from the CUTEst collection and their sizes.} 
\label{tab:uncpbs}
{\footnotesize
\begin{tabular}{lrlrlr}
\toprule
 Problem & $n$ & Problem & $n$ & Problem & $n$ \\
\midrule
ARGTRIGLS &   200 &   EIGENBLS &   110 &   MOREBV &  5000 \\
 CHNROSNB &    50 &   EIGENCLS &   462 & MSQRTALS &   529 \\
  COATING &   134 &   ERRINROS &    50 & MSQRTBLS &   529 \\
   COSINE & 10000 &   EXTROSNB &  1000 & NONCVXU2 & 10000 \\
DIXMAANE1 &  3000 &   FLETCHCR &  1000 & NONCVXUN & 10000 \\
 DIXMAANF &  9000 &   FMINSURF &  1024 & NONDQUAR & 10000 \\
 DIXMAANG &  9000 &   GENHUMPS &  5000 & SPMSRTLS & 10000 \\
 DIXMAANH &  9000 &    GENROSE &   500 & SSBRYBND &  5000 \\
 DIXMAANJ &  9000 & LUKSAN11LS &   100 & TQUARTIC &  5000 \\
 DIXMAANK &  9000 & LUKSAN21LS &   100 & & \\
 EIGENALS &   110 &   MODBEALE &  2000 & & \\
\bottomrule
\end{tabular}
}
\end{table}

Before comparing LMSD methods with the other gradient methods, we discuss the following two aspects of LMSD: the use of different decompositions of either $\bG$ or $\bS$ in LMSD-LYA, which has been presented in Section~\ref{sec:lya-ld}; the number of steps per sweep that are actually used by each LMSD method, in relation with the chosen memory parameter $m$. 

{\bf Different decompositions in LMSD-LYA.} In the quadratic case, we notice that there is not much difference between the listed decompositions to compute a basis for $\cals$. We repeat this experiment with LMSD-LYA, for general unconstrained problems, because the Hessian matrix is not constant during the iterations and therefore the way we discard the past gradients might be relevant. We recall that Cholesky decomposition (LMSD-LYA) discards the oldest gradients first, pivoted QR (LMSD-LYA-QR) selects the gradients in a different order; SVD (LMSD-LYA-SVD) takes a linear combination of the available gradients. For the last two methods, the tolerance to detect linear dependency is set to ${\sf thresh = 10^{-8}}$. 

Figure~\ref{fig:lya} shows the three decompositions for $m = 5$. Memory parameters $m=3,\,7$ are not reported as they are similar to the case $m = 5$. The conclusion is the same as for the quadratic case: the decomposition method does not seem to have a large impact on the performance of LMSD. However, for the general case, we remark that while LMSD-LYA solved all problems, both LMSD-LYA-QR and LMSD-LYA-SVD fail to solve one problem each, for all the tested memory parameters. In addition, LMSD-LYA seems more computationally efficient than the other methods. For these two reasons, we continue our analysis by focusing on Cholesky decomposition only. 

\begin{figure}[htb!]
 \centering
 \includegraphics[width=\textwidth]{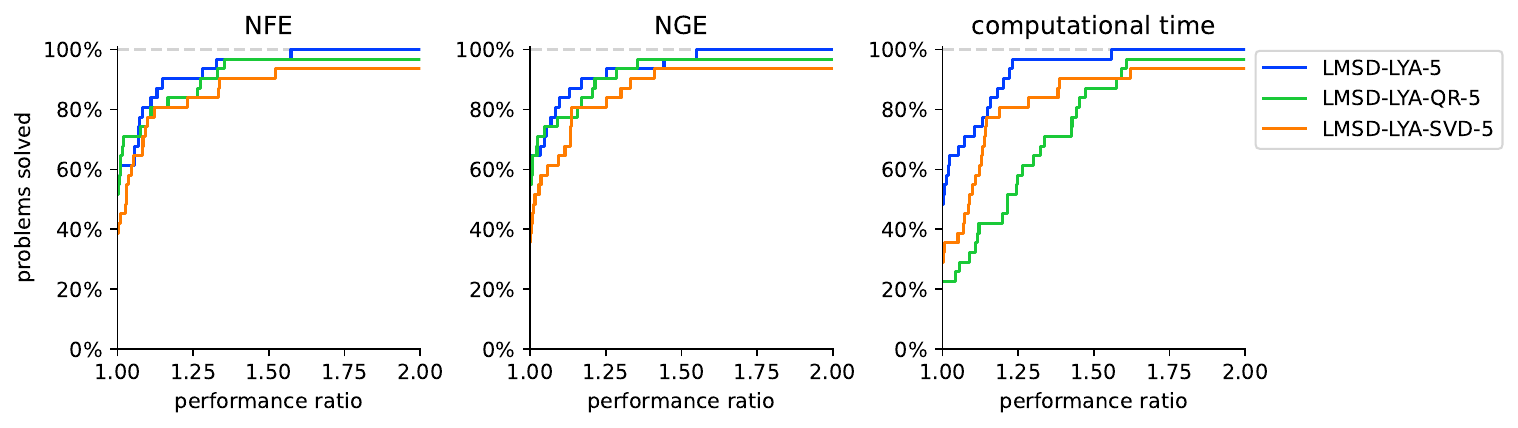}
 \caption{Performance profile for general unconstrained problems, based on the number of function evaluations, gradient evaluations, and computational time. Comparison between different decompositions for the matrix $\bG$ (or $\bS$) and $m = 5$.}
 \label{fig:lya}
\end{figure}

{\bf Average number of stepsizes per sweep.}
We quantify the efficiency of the various LMSD methods as follows. Ideally, each sweep should provide $m$ new stepsizes, which are supposed to be used in the next $m$ iterations. 
However, because of the algorithm we adopted, less than $m$ stepsizes are actually employed before the stack is cleared. For each problem and method, we compute the ratio between the number of iterations and the number of sweeps. This gives the average number of stepsizes that are used in each sweep. This value is in $[1,m]$, where the memory parameter
$m$ indicates the ideal situation where all the steps are used in a sweep. A method that often uses less than $m$ stepsizes might be inefficient, since the effort of computing new stepsizes (of approximately $\calo(m^2n)$ operations) is not entirely compensated. 

The number of iterations per sweep is shown in Figure~\ref{fig:sweeps} as a distribution function over the tested problems. An ideal curve should be a step function with a jump in $m$. For example, when $m = 3$, LMSD-CHOL, i.e., Fletcher's method, tends to use $3$ stepsizes on average for approximately $80\%$ of the problems; this is close to the desired situation. When $m = 5$, we notice that LMSD-H-LYA and LMSD-LYA have a similar behavior but for an average smaller than $5$. In all cases, LMSD-H-LYA is the curve that shows the lowest average number of steps per sweep. Another interesting behavior is the one of LMSD-PERT, which, for some problems, approaches the largest value $m$, but, for many others, shows a lower average. In $m=5,7$, more than $50\%$ of problems are solved by using only half of the available stepsizes per sweep. This behavior was reflected by the performance profiles of LMSD-PERT: while going from $m=5$ to $m=7$, we observed an improvement in terms of the number of function evaluations, but a deterioration in the computational time.

As $m$ increases, the deterioration of the average number of stepsizes per sweep is also visible for the other methods. As already remarked by \cite{fletcher2012limited,daniela2018steplength}, this suggests that choosing a large value for $m$ does not improve the LMSD methods for general unconstrained problems. This is in contrast with what we have observed in the quadratic case. 

\begin{figure}[htb!]
 \centering
 \includegraphics[width=\textwidth]{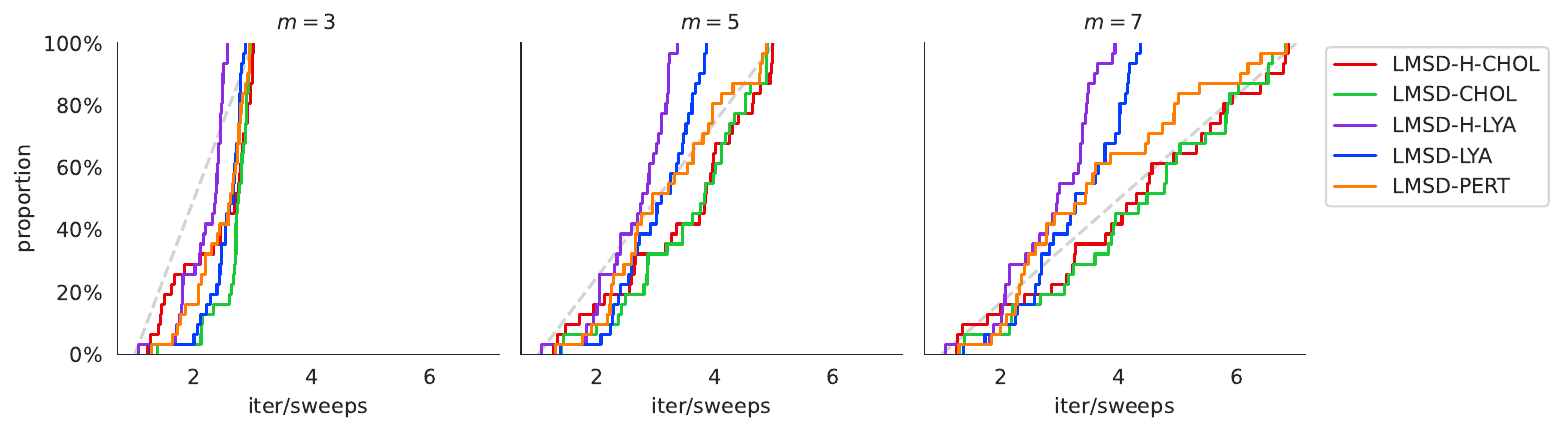}
 \caption{Cumulative distribution function of the number of iterations per sweep, i.e., the average number of stepsizes per sweep. Curves are based on the tested problems. Straight dashed lines indicate the uniform distribution over $[1,m]$.}
 \label{fig:sweeps}
\end{figure}

\begin{figure}[htb!]
 \centering
 \includegraphics[width=\textwidth]{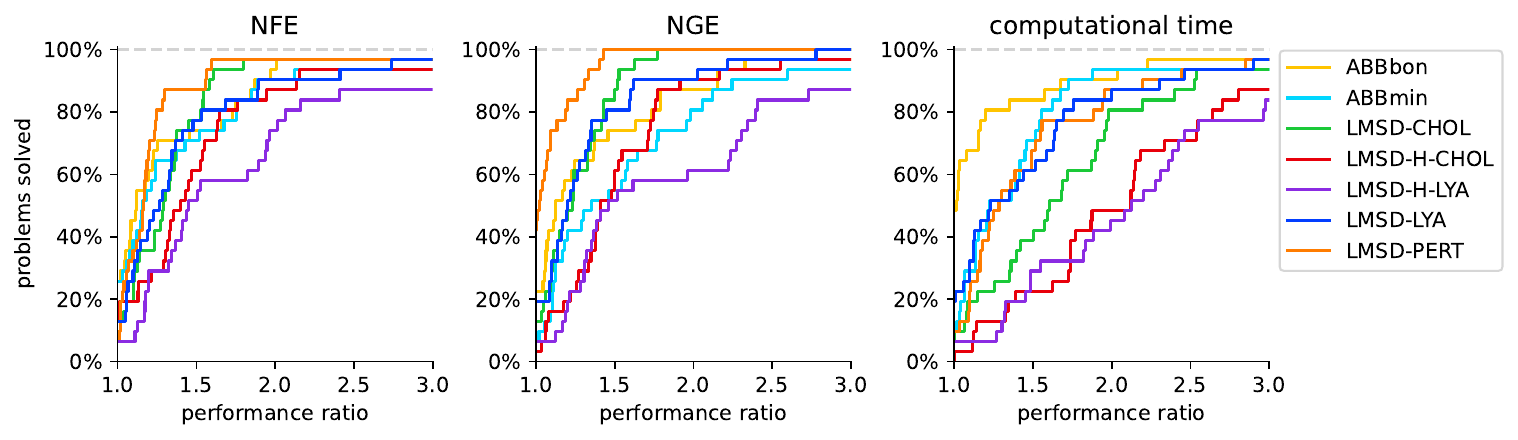}
 \caption{Performance profile for general unconstrained problems, based on the number of function evaluations, gradient evaluations, and computational time. Comparison between different ways to compute the new stepsizes of a sweep in LMSD, and the gradient method with nonmonotone line search and with either the $\abbmin$ or the $\abbbon$ step.}
 \label{fig:lmsdunc}
\end{figure}

{\bf Comparison with gradient methods.} In what follows, we consider only $m=5$ for the comparison with $\abbmin$ and $\abbbon$. For LMSD, we do not include $m = 3$ because it showed poorer results compared to the simpler nonmonotone gradient method. LMSD for $m = 7$ is not considered since it gives performances similar to $m = 5$, but with a higher computational cost. Results are shown in Figure~\ref{fig:lmsdunc}. From the performance profiles related to the computational time, we see that $\abbbon$ solves a high proportion of problems with the minimum computational time.

Regarding the performance profiles for both NFE and NGE, we note that LMSD-PERT has the highest curve; LMSD-LYA and LMSD-CHOL almost overlap for a performance ratio smaller than $1.5$; after that, the two curves split, and LMSD-CHOL reaches LMSD-PERT.

The LMSD-PERT method solves $42\%$ of the problems with the minimum number of gradient evaluations. By looking at the performance profile for the NFE, this fact does not seem to be complemented by a low number of function evaluations. Intuitively, this means that LMSD-PERT enters the backtracking procedure more often than the other methods, and it reflects what we have also observed in the central plot of Figure~\ref{fig:sweeps}. Any time we enter the backtracking procedure, the stack of stepsizes is cleared and the sweep is terminated. Then, the more backtracking we need, the smaller the number of stepsizes per sweep we use. 

LMSD-H-CHOL and LMSD-H-LYA, the ``harmonic" approaches, perform a little bit worse than the other methods: while LMSD-H-CHOL can still compete with $\abbmin$ and $\abbbon$ in terms of NFE and NGE, it performs worse in terms of computational time. 
LMSD-H-LYA performs generally worse than the other techniques; Figure~\ref{fig:sweeps} was already suggesting the poorer quality of the stepsizes of LMSD-H-LYA, which often need backtracking or lead to an increasing gradient norm. 

To complete the picture, Table~\ref{tab:lmsd5} reports two important quantities related to the performance profile: the proportion of problems solved by each method, and the proportion of problems solved with minimum cost, which is not always clearly visible from Figure~\ref{fig:lmsdunc}. We notice that $\abbmin$ and LMSD-H-LYA fail to solve one of the 31 tested problems. The $\abbbon$ stepsize solves $26\%$ of problems with minimum NFE and $48\%$ of problems with minimum computational time. LMSD-PERT wins in terms of NGE. The proportion of problems solved with minimum NFE is the same for LMSD-CHOL, LMSD-LYA, and LMSD-H-CHOL. 

\begin{table}[htb!]
\centering
\caption{For each method, we report the proportion of solved problems and the proportion of problems solved at minimum cost (performance ratio equal to 1) for different performance measures. The memory parameter is $m = 5$ for all the LMSD methods.}
\label{tab:lmsd5}
{\footnotesize
\begin{tabular}{lrrrr}
\toprule
Method & Solved (\%) & \multicolumn{3}{c}{PR = 1 (\%)} \\
 \ & \ & NFE & NGE & Time \\
\midrule
  $\abbbon$ &        1.00 &        0.26 & 0.23 & 0.48 \\
  LMSD-CHOL &        1.00 &        0.13 & 0.13 & 0.06 \\
LMSD-H-CHOL &        1.00 &        0.13 & 0.03 & 0.00 \\
   LMSD-LYA &        1.00 &        0.13 & 0.19 & 0.19 \\
  LMSD-PERT &        1.00 &        0.06 & 0.42 & 0.10 \\
  $\abbmin$ &        0.97 &        0.26 & 0.06 & 0.10 \\
 LMSD-H-LYA &        0.97 &        0.06 & 0.06 & 0.06 \\
\bottomrule
\end{tabular}
}
\end{table}

\section{Conclusions} \label{sec:con}
We have reviewed the limited memory steepest descent method proposed by Fletcher \cite{fletcher2012limited}, for both quadratic and general nonlinear unconstrained problems. In the context of strictly convex quadratic functions, we have explored pivoted QR and SVD as alternative ways to compute a basis for either the matrix $\bG$ (Ritz values) or $\bY$ (harmonic Ritz values). We have also proposed to improve the harmonic Ritz values by computing the Rayleigh quotients of their corresponding harmonic Ritz vectors.

Experiments in Section~\ref{sec:lmsd-expquad} have shown that the type of decomposition has little influence on the number of iterations of LMSD. The choice between Cholesky decomposition, pivoted QR and SVD is problem dependent. These three methods may compete with the $\abbmin$ and $\abbbon$ gradient methods.

The experiments also suggest that a larger memory parameter improves the performance of LMSD, and Ritz values seem to perform better than harmonic Ritz values. The modification of the harmonic Ritz values (Section~\ref{sec:harm}) effectively improves the number of iterations, at the extra expense of (relatively cheap) $\calo(m^3)$ work.

In the context of general nonlinear functions, we have given a theoretical foundation to Fletcher's idea \cite{fletcher2012limited} (LMSD-CHOL), by connecting the symmetrization of $\bT$ \eqref{eq:fletcherT} to a perturbation of $\bY$. We have proposed another LMSD method (LMSD-PERT) based on a different perturbation given by Schnabel \cite{schnabel1983quasi} in the area of quasi-Newton methods. An additional modification of LMSD for general functions (LMSD-LYA) has been obtained by adding symmetry constraints to the secant condition of LMSD for quadratic functions. The solution to this problem coincides with the solution to a Lyapunov equation.

In Section~\ref{sec:lmsd-expunc}, experiments on general unconstrained optimization problems have shown that, in contrast with the quadratic case, increasing the memory parameter does not necessarily improve the performance of LMSD. This may also be related to the choices made in Algorithm~\ref{algo:lmsd-unconstrained}, such as the sufficient decrease condition or the criteria to keep or discard old gradients. 

Given a certain memory parameter, the aforementioned LMSD methods seem to perform equally well in terms of the number of function evaluations and computational time. They all seem valid alternatives to the nonmonotone gradient method based on either $\abbmin$ or $\abbbon$ stepsizes, with the caveat that LMSD-PERT and LMSD-LYA tend to not exploit all the stepsizes computed in a sweep, more often than LMSD-CHOL. 

A Python code for the LMSD methods and the nonmonotone $\abbmin$ and $\abbbon$ is available at \href{https://github.com/gferrandi/lmsdpy}{github.com/gferrandi/lmsdpy}.

\medskip\noindent
\section{Declarations}
\subsection{Ethical approval}
Not applicable.
\subsection{Availability of supporting data}
The data used during the current study are available in the SuiteSparse Matrix Collection repository, \href{https://sparse.tamu.edu/}{sparse.tamu.edu}.
\subsection{Competing interests}
The authors have no competing interests to declare that are relevant to the content of this article.
\subsection{Funding}
This work has received funding from the European Union's Horizon 2020 research and innovation programme under the Marie Sklodowska-Curie grant agreement No 812912. 
\subsection{Authors' contributions}
All the authors equally contributed to all the manuscript parts.
\subsection{Acknowledgments}
We are grateful to the referees for their very useful suggestions which considerably improved the quality of the paper. We would also like to thank Nata\v{s}a Kreji\'c for the inspiring discussions on gradient methods.


\bibliography{references}


\begin{thebibliography}{32}
\ifx \bisbn   \undefined \def \bisbn  #1{ISBN #1}\fi
\ifx \binits  \undefined \def \binits#1{#1}\fi
\ifx \bauthor  \undefined \def \bauthor#1{#1}\fi
\ifx \batitle  \undefined \def \batitle#1{#1}\fi
\ifx \bjtitle  \undefined \def \bjtitle#1{#1}\fi
\ifx \bvolume  \undefined \def \bvolume#1{\textbf{#1}}\fi
\ifx \byear  \undefined \def \byear#1{#1}\fi
\ifx \bissue  \undefined \def \bissue#1{#1}\fi
\ifx \bfpage  \undefined \def \bfpage#1{#1}\fi
\ifx \blpage  \undefined \def \blpage #1{#1}\fi
\ifx \burl  \undefined \def \burl#1{\textsf{#1}}\fi
\ifx \doiurl  \undefined \def \doiurl#1{\url{https://doi.org/#1}}\fi
\ifx \betal  \undefined \def \betal{\textit{et al.}}\fi
\ifx \binstitute  \undefined \def \binstitute#1{#1}\fi
\ifx \binstitutionaled  \undefined \def \binstitutionaled#1{#1}\fi
\ifx \bctitle  \undefined \def \bctitle#1{#1}\fi
\ifx \beditor  \undefined \def \beditor#1{#1}\fi
\ifx \bpublisher  \undefined \def \bpublisher#1{#1}\fi
\ifx \bbtitle  \undefined \def \bbtitle#1{#1}\fi
\ifx \bedition  \undefined \def \bedition#1{#1}\fi
\ifx \bseriesno  \undefined \def \bseriesno#1{#1}\fi
\ifx \blocation  \undefined \def \blocation#1{#1}\fi
\ifx \bsertitle  \undefined \def \bsertitle#1{#1}\fi
\ifx \bsnm \undefined \def \bsnm#1{#1}\fi
\ifx \bsuffix \undefined \def \bsuffix#1{#1}\fi
\ifx \bparticle \undefined \def \bparticle#1{#1}\fi
\ifx \barticle \undefined \def \barticle#1{#1}\fi
\bibcommenthead
\ifx \bconfdate \undefined \def \bconfdate #1{#1}\fi
\ifx \botherref \undefined \def \botherref #1{#1}\fi
\ifx \url \undefined \def \url#1{\textsf{#1}}\fi
\ifx \bchapter \undefined \def \bchapter#1{#1}\fi
\ifx \bbook \undefined \def \bbook#1{#1}\fi
\ifx \bcomment \undefined \def \bcomment#1{#1}\fi
\ifx \oauthor \undefined \def \oauthor#1{#1}\fi
\ifx \citeauthoryear \undefined \def \citeauthoryear#1{#1}\fi
\ifx \endbibitem  \undefined \def \endbibitem {}\fi
\ifx \bconflocation  \undefined \def \bconflocation#1{#1}\fi
\ifx \arxivurl  \undefined \def \arxivurl#1{\textsf{#1}}\fi
\csname PreBibitemsHook\endcsname

\bibitem[\protect\citeauthoryear{Fletcher}{2012}]{fletcher2012limited}
\begin{barticle}
\bauthor{\bsnm{Fletcher}, \binits{R.}}:
\batitle{A limited memory steepest descent method}.
\bjtitle{Math. Program.}
\bvolume{135}(\bissue{1}),
\bfpage{413}--\blpage{436}
(\byear{2012})
\end{barticle}
\endbibitem

\bibitem[\protect\citeauthoryear{Di~Serafino
  et~al.}{2018}]{daniela2018steplength}
\begin{barticle}
\bauthor{\bsnm{Di~Serafino}, \binits{D.}},
\bauthor{\bsnm{Ruggiero}, \binits{V.}},
\bauthor{\bsnm{Toraldo}, \binits{G.}},
\bauthor{\bsnm{Zanni}, \binits{L.}}:
\batitle{On the steplength selection in gradient methods for unconstrained
  optimization}.
\bjtitle{Appl. Math. Comput.}
\bvolume{318},
\bfpage{176}--\blpage{195}
(\byear{2018})
\end{barticle}
\endbibitem

\bibitem[\protect\citeauthoryear{Zou and Magoul\`es}{2022}]{zou2022delayed}
\begin{barticle}
\bauthor{\bsnm{Zou}, \binits{Q.}},
\bauthor{\bsnm{Magoul\`es}, \binits{F.}}:
\batitle{Delayed gradient methods for symmetric and positive definite linear
  systems}.
\bjtitle{SIAM Rev.}
\bvolume{64}(\bissue{3}),
\bfpage{517}--\blpage{553}
(\byear{2022})
\end{barticle}
\endbibitem

\bibitem[\protect\citeauthoryear{Barzilai and Borwein}{1988}]{bb1988}
\begin{barticle}
\bauthor{\bsnm{Barzilai}, \binits{J.}},
\bauthor{\bsnm{Borwein}, \binits{J.M.}}:
\batitle{Two-point step size gradient methods}.
\bjtitle{IMA J. Numer. Anal.}
\bvolume{8}(\bissue{1}),
\bfpage{141}--\blpage{148}
(\byear{1988})
\end{barticle}
\endbibitem

\bibitem[\protect\citeauthoryear{Nocedal and
  Wright}{2006}]{nocedal2006numerical}
\begin{bbook}
\bauthor{\bsnm{Nocedal}, \binits{J.}},
\bauthor{\bsnm{Wright}, \binits{S.J.}}:
\bbtitle{Numerical Optimization},
\bedition{2nd} edn.
\bpublisher{Springer},
\blocation{New York, NY, USA}
(\byear{2006})
\end{bbook}
\endbibitem

\bibitem[\protect\citeauthoryear{Liu and Nocedal}{1989}]{liu1989limited}
\begin{barticle}
\bauthor{\bsnm{Liu}, \binits{D.C.}},
\bauthor{\bsnm{Nocedal}, \binits{J.}}:
\batitle{On the limited memory {BFGS} method for large scale optimization}.
\bjtitle{Math. Program.}
\bvolume{45},
\bfpage{503}--\blpage{528}
(\byear{1989})
\end{barticle}
\endbibitem

\bibitem[\protect\citeauthoryear{Raydan}{1997}]{raydan1997barzilai}
\begin{barticle}
\bauthor{\bsnm{Raydan}, \binits{M.}}:
\batitle{The {B}arzilai and {B}orwein gradient method for the large scale
  unconstrained minimization problem}.
\bjtitle{SIAM J. Optim.}
\bvolume{7}(\bissue{1}),
\bfpage{26}--\blpage{33}
(\byear{1997})
\end{barticle}
\endbibitem

\bibitem[\protect\citeauthoryear{Porta et~al.}{2015}]{porta2015}
\begin{barticle}
\bauthor{\bsnm{Porta}, \binits{F.}},
\bauthor{\bsnm{Prato}, \binits{M.}},
\bauthor{\bsnm{Zanni}, \binits{L.}}:
\batitle{A new steplength selection for scaled gradient methods with
  application to image deblurring}.
\bjtitle{J. Sci. Comp.}
\bvolume{65}(\bissue{3}),
\bfpage{895}--\blpage{919}
(\byear{2015})
\end{barticle}
\endbibitem

\bibitem[\protect\citeauthoryear{Franchini et~al.}{2020}]{franchini2020ritz}
\begin{barticle}
\bauthor{\bsnm{Franchini}, \binits{G.}},
\bauthor{\bsnm{Ruggiero}, \binits{V.}},
\bauthor{\bsnm{Zanni}, \binits{L.}}:
\batitle{Ritz-like values in steplength selections for stochastic gradient
  methods}.
\bjtitle{Soft Comput.}
\bvolume{24}(\bissue{23}),
\bfpage{17573}--\blpage{17588}
(\byear{2020})
\end{barticle}
\endbibitem

\bibitem[\protect\citeauthoryear{Crisci et~al.}{2023}]{crisci2023hybrid}
\begin{barticle}
\bauthor{\bsnm{Crisci}, \binits{S.}},
\bauthor{\bsnm{Porta}, \binits{F.}},
\bauthor{\bsnm{Ruggiero}, \binits{V.}},
\bauthor{\bsnm{Zanni}, \binits{L.}}:
\batitle{Hybrid limited memory gradient projection methods for box-constrained
  optimization problems}.
\bjtitle{Comput. Optim. Appl.}
\bvolume{84}(\bissue{1}),
\bfpage{151}--\blpage{189}
(\byear{2023})
\end{barticle}
\endbibitem

\bibitem[\protect\citeauthoryear{Fukaya et~al.}{2020}]{FKN20}
\begin{barticle}
\bauthor{\bsnm{Fukaya}, \binits{T.}},
\bauthor{\bsnm{Kannan}, \binits{R.}},
\bauthor{\bsnm{Nakatsukasa}, \binits{Y.}},
\bauthor{\bsnm{Yamamoto}, \binits{Y.}},
\bauthor{\bsnm{Yanagisawa}, \binits{Y.}}:
\batitle{{Shifted Cholesky QR for computing the QR factorization of
  ill-conditioned matrices}}.
\bjtitle{SIAM J. Sci. Comput.}
\bvolume{42}(\bissue{1}),
\bfpage{477}--\blpage{503}
(\byear{2020})
\end{barticle}
\endbibitem

\bibitem[\protect\citeauthoryear{Gu and Eisenstat}{1996}]{gu1996efficient}
\begin{barticle}
\bauthor{\bsnm{Gu}, \binits{M.}},
\bauthor{\bsnm{Eisenstat}, \binits{S.C.}}:
\batitle{Efficient algorithms for computing a strong rank-revealing {QR}
  factorization}.
\bjtitle{SIAM J. Sci. Comput.}
\bvolume{17}(\bissue{4}),
\bfpage{848}--\blpage{869}
(\byear{1996})
\end{barticle}
\endbibitem

\bibitem[\protect\citeauthoryear{Schnabel}{1983}]{schnabel1983quasi}
\begin{botherref}
\oauthor{\bsnm{Schnabel}, \binits{R.B.}}:
Quasi-{N}ewton methods using multiple secant equations.
Technical Report CU-CS-247-83,
Department of Computer Science, University of Colorado,
Boulder, USA
(1983)
\end{botherref}
\endbibitem

\bibitem[\protect\citeauthoryear{Simoncini}{2016}]{simoncini2016computational}
\begin{barticle}
\bauthor{\bsnm{Simoncini}, \binits{V.}}:
\batitle{Computational methods for linear matrix equations}.
\bjtitle{SIAM Rev.}
\bvolume{58}(\bissue{3}),
\bfpage{377}--\blpage{441}
(\byear{2016})
\end{barticle}
\endbibitem

\bibitem[\protect\citeauthoryear{Curtis and Guo}{2016}]{curtis2016handling}
\begin{barticle}
\bauthor{\bsnm{Curtis}, \binits{F.E.}},
\bauthor{\bsnm{Guo}, \binits{W.}}:
\batitle{Handling nonpositive curvature in a limited memory steepest descent
  method}.
\bjtitle{IMA J. Numer. Anal.}
\bvolume{36}(\bissue{2}),
\bfpage{717}--\blpage{742}
(\byear{2016})
\end{barticle}
\endbibitem

\bibitem[\protect\citeauthoryear{Curtis and Guo}{2018}]{curtis2018linear}
\begin{barticle}
\bauthor{\bsnm{Curtis}, \binits{F.E.}},
\bauthor{\bsnm{Guo}, \binits{W.}}:
\batitle{{R}-linear convergence of limited memory steepest descent}.
\bjtitle{IMA J. Numer. Anal.}
\bvolume{38}(\bissue{2}),
\bfpage{720}--\blpage{742}
(\byear{2018})
\end{barticle}
\endbibitem

\bibitem[\protect\citeauthoryear{Morgan}{1991}]{morgan1991computing}
\begin{barticle}
\bauthor{\bsnm{Morgan}, \binits{R.B.}}:
\batitle{Computing interior eigenvalues of large matrices}.
\bjtitle{Linear Algebra Appl.}
\bvolume{154},
\bfpage{289}--\blpage{309}
(\byear{1991})
\end{barticle}
\endbibitem

\bibitem[\protect\citeauthoryear{Sleijpen and van~den
  Eshof}{2003}]{sleijpen2003use}
\begin{barticle}
\bauthor{\bsnm{Sleijpen}, \binits{G.L.G.}},
\bauthor{\bsnm{Eshof}, \binits{J.}}:
\batitle{On the use of harmonic {R}itz pairs in approximating internal
  eigenpairs}.
\bjtitle{Linear Algebra Appl.}
\bvolume{358}(\bissue{1-3}),
\bfpage{115}--\blpage{137}
(\byear{2003})
\end{barticle}
\endbibitem

\bibitem[\protect\citeauthoryear{Parlett}{1998}]{Par98}
\begin{bbook}
\bauthor{\bsnm{Parlett}, \binits{B.N.}}:
\bbtitle{The {S}ymmetric {E}igenvalue {P}roblem}.
\bpublisher{SIAM},
\blocation{Philadelphia, PA}
(\byear{1998})
\end{bbook}
\endbibitem

\bibitem[\protect\citeauthoryear{Beattie}{1998}]{beattie1998harmonic}
\begin{barticle}
\bauthor{\bsnm{Beattie}, \binits{C.}}:
\batitle{Harmonic {R}itz and {L}ehmann bounds}.
\bjtitle{Electron. Trans. Numer. Anal}
\bvolume{7},
\bfpage{18}--\blpage{39}
(\byear{1998})
\end{barticle}
\endbibitem

\bibitem[\protect\citeauthoryear{Fletcher}{1990}]{fletcher1990low}
\begin{barticle}
\bauthor{\bsnm{Fletcher}, \binits{R.}}:
\batitle{Low storage methods for unconstrained optimization}.
\bjtitle{Lect. Appl. Math. (AMS)}
\bvolume{26},
\bfpage{165}--\blpage{179}
(\byear{1990})
\end{barticle}
\endbibitem

\bibitem[\protect\citeauthoryear{Christof}{2010}]{vomel2010note}
\begin{barticle}
\bauthor{\bsnm{Christof}, \binits{V.}}:
\batitle{A note on harmonic {R}itz values and their reciprocals}.
\bjtitle{Numer. Linear Algebra Appl.}
\bvolume{17}(\bissue{1}),
\bfpage{97}--\blpage{108}
(\byear{2010})
\end{barticle}
\endbibitem

\bibitem[\protect\citeauthoryear{Yasuda and Hirai}{1979}]{yasuda1979upper}
\begin{barticle}
\bauthor{\bsnm{Yasuda}, \binits{K.}},
\bauthor{\bsnm{Hirai}, \binits{K.}}:
\batitle{Upper and lower bounds on the solution of the algebraic {R}iccati
  equation}.
\bjtitle{IEEE Trans. Automat. Contr.}
\bvolume{24}(\bissue{3}),
\bfpage{483}--\blpage{487}
(\byear{1979})
\end{barticle}
\endbibitem

\bibitem[\protect\citeauthoryear{Grippo et~al.}{1986}]{grippo1986nonmonotone}
\begin{barticle}
\bauthor{\bsnm{Grippo}, \binits{L.}},
\bauthor{\bsnm{Lampariello}, \binits{F.}},
\bauthor{\bsnm{Lucidi}, \binits{S.}}:
\batitle{{A nonmonotone line search technique for Newton's method}}.
\bjtitle{SIAM J. Numer. Anal.}
\bvolume{23}(\bissue{4}),
\bfpage{707}--\blpage{716}
(\byear{1986})
\end{barticle}
\endbibitem

\bibitem[\protect\citeauthoryear{Frassoldati et~al.}{2008}]{frassoldati2008new}
\begin{barticle}
\bauthor{\bsnm{Frassoldati}, \binits{G.}},
\bauthor{\bsnm{Zanni}, \binits{L.}},
\bauthor{\bsnm{Zanghirati}, \binits{G.}}:
\batitle{New adaptive stepsize selections in gradient methods}.
\bjtitle{J. Ind. Manag.}
\bvolume{4}(\bissue{2}),
\bfpage{299}
(\byear{2008})
\end{barticle}
\endbibitem

\bibitem[\protect\citeauthoryear{Bonettini et~al.}{2008}]{bonettini2008scaled}
\begin{barticle}
\bauthor{\bsnm{Bonettini}, \binits{S.}},
\bauthor{\bsnm{Zanella}, \binits{R.}},
\bauthor{\bsnm{Zanni}, \binits{L.}}:
\batitle{A scaled gradient projection method for constrained image deblurring}.
\bjtitle{Inverse Probl.}
\bvolume{25}(\bissue{1}),
\bfpage{015002}
(\byear{2008})
\end{barticle}
\endbibitem

\bibitem[\protect\citeauthoryear{Byrd et~al.}{1995}]{byrd1995limited}
\begin{barticle}
\bauthor{\bsnm{Byrd}, \binits{R.H.}},
\bauthor{\bsnm{Lu}, \binits{P.}},
\bauthor{\bsnm{Nocedal}, \binits{J.}},
\bauthor{\bsnm{Zhu}, \binits{C.}}:
\batitle{A limited memory algorithm for bound constrained optimization}.
\bjtitle{SIAM J. Sci. Comput.}
\bvolume{16}(\bissue{5}),
\bfpage{1190}--\blpage{1208}
(\byear{1995})
\end{barticle}
\endbibitem

\bibitem[\protect\citeauthoryear{Davis and Hu}{2011}]{suitesparse}
\begin{barticle}
\bauthor{\bsnm{Davis}, \binits{T.A.}},
\bauthor{\bsnm{Hu}, \binits{Y.}}:
\batitle{The {U}niversity of {F}lorida sparse matrix collection}.
\bjtitle{ACM Trans. Math. Softw.}
\bvolume{38}(\bissue{1}),
\bfpage{1}--\blpage{25}
(\byear{2011})
\end{barticle}
\endbibitem

\bibitem[\protect\citeauthoryear{Ferrandi et~al.}{2023}]{FHK23}
\begin{barticle}
\bauthor{\bsnm{Ferrandi}, \binits{G.}},
\bauthor{\bsnm{Hochstenbach}, \binits{M.E.}},
\bauthor{\bsnm{Kreji\'c}, \binits{N.}}:
\batitle{A harmonic framework for stepsize selection in gradient methods}.
\bjtitle{Comput. Opt. Appl.}
\bvolume{85},
\bfpage{75}--\blpage{106}
(\byear{2023})
\end{barticle}
\endbibitem

\bibitem[\protect\citeauthoryear{Dolan and Mor{\'e}}{2002}]{perfprofile}
\begin{barticle}
\bauthor{\bsnm{Dolan}, \binits{E.D.}},
\bauthor{\bsnm{Mor{\'e}}, \binits{J.J.}}:
\batitle{Benchmarking optimization software with performance profiles}.
\bjtitle{Math. Program.}
\bvolume{91}(\bissue{2}),
\bfpage{201}--\blpage{213}
(\byear{2002})
\end{barticle}
\endbibitem

\bibitem[\protect\citeauthoryear{Fowkes et~al.}{2022}]{pycutest2022}
\begin{barticle}
\bauthor{\bsnm{Fowkes}, \binits{J.}},
\bauthor{\bsnm{Roberts}, \binits{R.}},
\bauthor{\bsnm{B\H{u}rmen}, \binits{A.}}:
\batitle{Py{CUTE}st: an open source {P}ython package of optimization test
  problems}.
\bjtitle{J. Open Source Softw.}
\bvolume{7}(\bissue{78}),
\bfpage{4377}
(\byear{2022})
\end{barticle}
\endbibitem

\bibitem[\protect\citeauthoryear{Gould et~al.}{2015}]{cutest}
\begin{barticle}
\bauthor{\bsnm{Gould}, \binits{N.I.M.}},
\bauthor{\bsnm{Orban}, \binits{D.}},
\bauthor{\bsnm{Toint}, \binits{P.L.}}:
\batitle{{CUTE}st: a constrained and unconstrained testing environment with
  safe threads for mathematical optimization}.
\bjtitle{Comput. Optim. Appl.}
\bvolume{60},
\bfpage{545}--\blpage{557}
(\byear{2015})
\end{barticle}
\endbibitem

\end{thebibliography}
\end{document}